\documentclass[12pt]{amsart}
\usepackage{amssymb,latexsym}
\usepackage{enumerate}
\usepackage{amsmath, amsthm, amsfonts, amssymb, mathrsfs}
\usepackage{verbatim}
\usepackage{hyperref}
\usepackage{color}

\begin{document}

% define theorem environments
\newtheorem{theorem}{Theorem}    %[section]
\newtheorem{proposition}[theorem]{Proposition}
\newtheorem{conjecture}[theorem]{Conjecture}
\def\theconjecture{\unskip}
\newtheorem{corollary}[theorem]{Corollary}
\newtheorem{lemma}[theorem]{Lemma}
\newtheorem{sublemma}[theorem]{Sublemma}
\newtheorem{observation}[theorem]{Observation}
\theoremstyle{definition}
\newtheorem{definition}{Definition}
\newtheorem{notation}[definition]{Notation}
\newtheorem{remark}[definition]{Remark}
\newtheorem{question}[definition]{Question}
\newtheorem{questions}[definition]{Questions}
\newtheorem{example}[definition]{Example}
\newtheorem{problem}[definition]{Problem}
\newtheorem{exercise}[definition]{Exercise}

\numberwithin{theorem}{section} \numberwithin{definition}{section}
\numberwithin{equation}{section}

\def\earrow{{\mathbf e}}
\def\rarrow{{\mathbf r}}
\def\uarrow{{\mathbf u}}
\def\varrow{{\mathbf V}}
\def\tpar{T_{\rm par}}
\def\apar{A_{\rm par}}

\def\reals{{\mathbb R}}
\def\torus{{\mathbb T}}
\def\heis{{\mathbb H}}
\def\integers{{\mathbb Z}}
\def\naturals{{\mathbb N}}
\def\complex{{\mathbb C}\/}
\def\distance{\operatorname{distance}\,}
\def\support{\operatorname{support}\,}
\def\dist{\operatorname{dist}\,}
\def\Span{\operatorname{span}\,}
\def\degree{\operatorname{degree}\,}
\def\kernel{\operatorname{kernel}\,}
\def\dim{\operatorname{dim}\,}
\def\codim{\operatorname{codim}}
\def\trace{\operatorname{trace\,}}
\def\Span{\operatorname{span}\,}
\def\dimension{\operatorname{dimension}\,}
\def\codimension{\operatorname{codimension}\,}
\def\nullspace{\scriptk}
\def\kernel{\operatorname{Ker}}
\def\ZZ{ {\mathbb Z} }
\def\p{\partial}
\def\rp{{ ^{-1} }}
\def\Re{\operatorname{Re\,} }
\def\Im{\operatorname{Im\,} }
\def\ov{\overline}
\def\eps{\varepsilon}
\def\lt{L^2}
\def\diver{\operatorname{div}}
\def\curl{\operatorname{curl}}
\def\etta{\eta}
\newcommand{\norm}[1]{ \|  #1 \|}
\def\expect{\mathbb E}
\def\bull{$\bullet$\ }
\newtheorem{pro}{Proposition}[section]
\def\xone{x_1}
\def\xtwo{x_2}
\def\xq{x_2+x_1^2}
\newcommand{\abr}[1]{ \langle  #1 \rangle}
\allowdisplaybreaks

\newcommand{\Norm}[1]{ \left\|  #1 \bigg\| }
\newcommand{\set}[1]{ \left\{ #1 \bigg\} }
\def\one{\mathbf 1}
\def\whole{\mathbf V}
\newcommand{\modulo}[2]{[#1]_{#2}}

\def\scriptf{{\mathcal F}}
\def\scriptg{{\mathcal G}}
\def\scriptm{{\mathcal M}}
\def\scriptb{{\mathcal B}}
\def\scriptc{{\mathcal C}}
\def\scriptt{{\mathcal T}}
\def\scripti{{\mathcal I}}
\def\scripte{{\mathcal E}}
\def\scriptv{{\mathcal V}}
\def\scriptw{{\mathcal W}}
\def\scriptu{{\mathcal U}}
\def\scriptS{{\mathcal S}}
\def\scripta{{\mathcal A}}
\def\scriptr{{\mathcal R}}
\def\scripto{{\mathcal O}}
\def\scripth{{\mathcal H}}
\def\scriptd{{\mathcal D}}
\def\scriptl{{\mathcal L}}
\def\scriptn{{\mathcal N}}
\def\scriptp{{\mathcal P}}
\def\scriptk{{\mathcal K}}
\def\frakv{{\mathfrak V}}

\begin{comment}
	\def\scriptx{{\mathcal X}}
	\def\scriptj{{\mathcal J}}
	\def\scriptr{{\mathcal R}}
	\def\scriptS{{\mathcal S}}
	\def\scripta{{\mathcal A}}
	\def\scriptk{{\mathcal K}}
	\def\scriptp{{\mathcal P}}
	\def\frakg{{\mathfrak g}}
	\def\frakG{{\mathfrak G}}
	\def\boldn{\mathbf N}
\end{comment}

\author{Hui Li}
\address{Hui Li\\
	School of Mathematics and Information Science\\
	Henan Polytechnic University\\
	Jiaozuo 454000\\
	People's Republic of China} \email{huili@hpu.edu.cn}

\author{Ziyuan Lian}
\address{Ziyuan Lian\\
	School of Mathematics and Information Science\\
	Henan Polytechnic University\\
	Jiaozuo 454000\\
	People's Republic of China} \email{zylian@home.hpu.edu.cn}

\author{Honghai Liu}
\address{Honghai Liu\\
	School of Mathematics and Information Science\\
	Henan Polytechnic University\\
	Jiaozuo 454000\\
	People's Republic of China} \email{hhliu@hpu.edu.cn}

\author{Zengyan Si}
\address{Zengyan Si\\
	School of Mathematics and Information Science\\
	Henan Polytechnic University\\
	Jiaozuo 454000\\
	People's Republic of China} \email{zengyan@hpu.edu.cn}

\author{Jie Wang}
\address{
	Jie Wang\\
	School of Mathematical Sciences\\
	University of Chinese Academy of Sciences\\
	Beijing 100049\\People's
	Republic of China} \email{jiewang@home.hpu.edu.cn}

\thanks{ The work is supported by Natural Science Foundation of China (No.12526417, No.12571101), the Fundamental Research Funds for the Universities of Henan Province (No.NSFRF2502089) and Postgraduate Education Reform and Quality Improvement
Project of Henan Province (No.YJS2026YBGZZ10).}


\date{}

\title[On $L^p$ bounds for the Littlewood-Paley $g_\alpha$ function]{On $L^p$ bounds for the Littlewood-Paley $g_\alpha$ function with fractional poisson kernel}

\maketitle

\begin{abstract}
Let \(0<\alpha\leq1\), and let \(g_\alpha\) be the Littlewood-Paley function associated with the fractional Poisson kernel, which reduces to the classical \(g\)-function when \(\alpha=1\). We establish dimension-free \(L^p\) bounds for \(g_\alpha\) for every \(1<p<\infty\) and identify its exact \(L^2\) norm as \(\sqrt{\alpha/2}\). We further determine the precise asymptotic behavior of the weak-type \((1,1)\) constant as \(n\to\infty\), yielding quantitative upper and lower bounds in terms of the dimension.
\end{abstract}

\section{Introduction}

For $f \in \mathscr{S}(\mathbb R^n)$, the class of Schwartz function on $\mathbb R^n$, the $\alpha$-harmonic extension $u = u(x,t)$ is defined as the unique solution to the boundary value problem:
\[
\begin{cases}
	\partial_t^2 u + \dfrac{1-\alpha}{t}\partial_t u + \Delta_x u = 0, & (x,t)\in \mathbb R_+^{n+1} := \mathbb R^n \times (0,\infty), \\[4pt]
	u(x,0) = f(x), & x\in \mathbb R^n.
\end{cases}
\]
 This extension is realized explicitly via the fractional Poisson kernel
\[
P_t^\alpha(x) = \frac{c_{n,\alpha} t^\alpha}{(t^2 + |x|^2)^{\frac{n+\alpha}{2}}},
\qquad
c_{n,\alpha} = \frac{\Gamma((n+\alpha)/2)}{\pi^{n/2}\Gamma(\alpha/2)},
\]
so that
\[
u(x,t) = (P_t^\alpha * f)(x) = c_{n,\alpha} \int_{\mathbb R^n} \frac{t^\alpha }{(|x-y|^2 + t^2)^{\frac{n+\alpha}{2}}}f(y)\,dy.
\]
The normalization constant $c_{n,\alpha}$ is chosen so that $\int_{\mathbb R^n} P_t^\alpha(x)\,dx = 1$, ensuring that the extension preserves constants.

%The interplay between fractional differential operators and harmonic analysis has witnessed remarkable progress over the past two decades, largely due to the fundamental extension method introduced by Caffarelli and Silvestre \cite{CS}. Their seminal work established that the fractional Laplacian $(-\Delta)^{\alpha/2}$, $0<\alpha<2$, can be realized as the Dirichlet-to-Neumann map of a degenerate elliptic equation in one higher dimension. This insight has revolutionized the study of nonlocal operators, providing a powerful local framework for problems ranging from fractional Sobolev spaces to stochastic processes and free boundary problems.
The extension method, originally introduced by Caffarelli and Silvestre \cite{CS}, has found diverse applications in analysis and PDEs. In the study of fractional Schr\"odinger equations, Ambrosio and Hajaiej \cite{AH} employed the $\alpha$-harmonic extension to establish the existence of multiple positive solutions via variational methods, exploiting the local structure of the extended problem to overcome the nonlocality of the fractional Laplacian. The technique has also been instrumental in proving regularity and stability estimates for fractional elliptic equations, as demonstrated in the works of Cabr\'e and Tan \cite{CT}, and later extended to nonlinear settings by numerous authors. In harmonic analysis, the $\alpha$-harmonic extension is intimately connected with Poisson semigroups and Littlewood--Paley theory, providing a bridge between fractional operators and classical square functions.

A fundamental result in classical Littlewood--Paley theory asserts that, for every $1<p<\infty$, there is a constant $C_p$ such that
\begin{equation}\label{eq:stein-classical}
\norm{g(f)}_{L^p(\mathbb R^n)}\le C_p\norm{f}_{L^p(\mathbb R^n)},
\end{equation}
with $C_p$ independent of the dimension $n$. For $0<\alpha\leq 1$, the fractional Littlewood--Paley $g_\alpha$ function associated with the fractional Poisson kernel $P_t^\alpha$ is defined by
\begin{equation}\label{eq: fractional}
g_\alpha(f)(x) = \left( \int_0^\infty |\nabla u(x,t)|^2 \, t \, dt \right)^{1/2}.
\end{equation}

In the case $\alpha=1$, $g_\alpha$ coincides with the classical Littlewood--Paley $g$-function, which is one of the most fundamental objects in harmonic analysis, serving as a quantitative measure of oscillation and a key tool in the study of function spaces. In the classical setting, Stein \cite{Stein70, Stein71} established that the $g$-function is bounded on $L^p(\mathbb R^n)$ for all $1<p<\infty$ and is of weak type $(1,1)$. The study of fractional square functions, particularly in the context of semigroups, has its roots in the work of Segovia and Wheeden\cite{Segovia}. They introduced a notion of fractional derivative and used it to develop a theory of Euclidean square functions for potential spaces. This idea was later adapted to a general semigroup setting by Betancor et al.\cite{Betancor}, who used these fractional derivatives to characterize potential spaces for the Ornstein-Uhlenbeck operator. More recently, the theory of square functions has been instrumental in providing new characterizations of fractional Sobolev spaces. Sato, Wang, Yang, and Yuan \cite{sato} established generalized Littlewood-Paley characterizations of fractional Sobolev spaces. For recent advances and applications of fractional square functions, we refer the reader to\cite{bao,cao,wang} and the references therein.

Dimension-free estimates are essential for extending results from finite to infinite dimensions.These results reveal a rich structure where certain operators exhibit bounds independent of the dimension, while others grow at a controlled rate. This research was pioneered by Stein\cite{Stein82}, who obtained dimension-free bounds on $L^p(\mathbb{R}^n)(1 < p \leq \infty)$ for the centered Hardy-Littlewood maximal function over Euclidean balls. Notable contributions include Stein and Str\"omberg \cite{SS}, Bourgain \cite{Bourgain86a, Bourgain86b, Bourgain14}, and Bourgain, Mirek, Stein, and Wr\'obel \cite{BMSW18, BMSW19, BMSW20}. In particular, Stein \cite{Stein86} proved that, for $1<p<\infty,$ the $L^p$ bounds of the classical Littlewood--Paley $g$ function and Riesz transform are independent of the dimension $n$. For the case $p=1$, Janakiraman\cite{Janakiraman} showed that the weak type $(1,1)$ bound of $R_j$ (a special singular integral) is at worst $O(\log n)$. Recently, Lai\cite{Lai} showed the weak type  $(1,1)$ bound of the classical Littlewood--Paley $g$ function is at worst $O(n^3)$. On the other hand, Janakiraman\cite{Janakiraman2} investigated the limiting weak-type behavior of the Hardy-Littlewood maximal function $M$, providing a new approach to determining lower bounds for the optimal constants of $M$, as well as for other operators including singular integrals and fractional integral operators. For further developments, see \cite{Ding1, Ding2, Guo, Hou} and the references cited therein. In this paper, we establish dimension-free \(L^p\) estimates for \(g_\alpha\) for every \(1<p<\infty\), with exact \(L^2\) norm \(\sqrt{\alpha/2}\). We then study the dimensional dependence of the weak-type \((1,1)\) constant and determine its precise asymptotic behavior as \(n\to\infty\).
 Our main results are as follows.

% recent years have witnessed growing interest in dimension-independent estimates for classical operators in harmonic analysis. This line of research was motivated by questions in high-dimensional analysis and by the desire to understand the dependence of operator norms on the ambient dimension.

%This result is somewhat surprising, as it shows that despite the fractional nature of the kernel and the presence of the full gradient in the definition of $\mathfrak{g}_\alpha$, the $L^2$-operator norm is an absolute constant depending only on $\alpha$. The proof exploits the explicit Fourier multiplier representation of the fractional Poisson kernel and the integral identities for modified Bessel functions.

\begin{theorem}[Dimension-free estimate]\label{thm:main}
Let $1<p<\infty$, $n\ge1$, and $0<\alpha\le1$. Then, for every $f\in L^p(\mathbb R^n)$,
\begin{equation*}\label{eq:main-Lp}
||g_\alpha(f)||_{L^p(\mathbb R^n)}\le C_p||f||_{L^p(\mathbb R^n)},
\end{equation*}
where the constant $C_p$ is independent of both $n$ and $\alpha$.
\end{theorem}

\begin{theorem}[$L^2$ identities with explicit constant]\label{thm:L2}
	For $f \in L^2(\mathbb R^n)$ and $0<\alpha\leq 1$, we have
	\[
	\|g_\alpha(f)\|_{L^2(\mathbb R^n)} = \sqrt{\frac{\alpha}{2}} \|f\|_{L^2(\mathbb R^n)}.
	\]
	\end{theorem}

Next, we prove that the weak $(1,1)$ bound of $g_\alpha$ is at worst $O(n^2)$. For $\alpha=1$, this improves the previous bound $O(n^3)$ for $g$ \cite[Theorem1.1]{Lai} to $O(n^2)$.

\begin{theorem}[The endpoint estimate]\label{thm:weak11}
	For $0<\alpha \leq 1$, there exists a constant $C_\alpha>0$, independent of $n$, such that for every $f\in L^1(\mathbb R^n)$ and every $\lambda>0$,
	$$
	\lambda \left|\{x\in \mathbb R^n : g_\alpha(f)(x) > \lambda\}\right|
	\lesssim C_\alpha  n^2 \|f\|_{L^1(\mathbb R^n)}.$$
\end{theorem}

%The proof relies on a modified Calder\'on--Zygmund decomposition due to Janakiraman \cite{Janakiraman} and careful estimates of the gradient kernel on the exceptional set. The exponent $7/2$ emerges from the geometric properties of the enlarged cubes and the pointwise bounds on the kernel differences.

We determine the precise limiting weak-type behavior as the level parameter tends to zero. This provides a sharp asymptotic for the distribution function of $g_\alpha(f)$ at small scales.
%The constant involves the dimension in a nontrivial way through the product $\prod_{k=1}^{n-1}(k+\alpha)$, which reflects the fractional nature of the kernel.

\begin{theorem}[Limiting weak-type behavior]\label{thm:limit}
	For $0<\alpha  \leq 1$ and $f\in L^1(\mathbb R^n)$, we have
	\[
	\lim_{\lambda \to 0^+} \lambda \left|\{x\in \mathbb R^n : g_\alpha(f)(x) > \lambda\}\right|
	= \frac{c_{n,\alpha} \omega_{n-1}}{n}
	\sqrt{\frac{n!}{2\prod_{k=1}^{n-1}(k+\alpha)}}
	\left|\int_{\mathbb R^n} f(x)\,dx\right|,
	\]
	where $\omega_{n-1}= 2\pi^{n/2}/\Gamma(n/2)$ denotes the surface area of the unit sphere in $\mathbb R^n$.
\end{theorem}

%\begin{corollary}\label{c1}
%	For $0<\alpha  \leq 1$ and $f\in L^1(\mathbb R^n)$, we have
%	\begin{align*}
%		\inf\sup_{\lambda>0}\lambda\left|\{x\in \mathbb{R}^n : \mathfrak{g}_\alpha(f)(x) > \lambda\}\right|
%		\geq c_{n,\alpha} \omega_{n-1} \sqrt{\frac{n!}{2\prod_{k=1}^{n-1}(k+\alpha)}},
%	\end{align*}
%	where the infimum is taken over all functions $f\in L^1(\mathbb R^n)$ with $\|f\|_{L^1(\mathbb R^n)}=1$.
%\end{corollary}
\begin{remark}
It is known from the work of Stein \cite{Stein70,Stein86} that, for $1<p<\infty,$ the $L^p$ bound of the classical Littlewood--Paley $g$ function is independent of the underlying dimension. In contrast to the Poisson semigroup $\{P_t\}_{t\geq 0}$,
$\{P_t^\alpha \}_{t\geq 0}$ does not satisfy the semigroup law, and hence the standard methods do not readily extend to fractional square functions. The main observation of this paper is instead a positive representation of $P_t^\alpha$ as an average of classical Poisson kernels at dilated scales. Because the averaging measure is a probability measure supported on $[1,\infty)$ and is independent of $n$, the full-gradient square function is pointwise dominated by the classical one.
\end{remark}
\begin{remark}
When $\alpha=1$, $g_\alpha$ reduces to the classical Littlewood--Paley $g$-function. The $L^2$-norm of $g$ is  $\sqrt{\frac{1}{2}}$, which can be obtained directly via the Fourier transform, thanks to the well-behaved nature of the Poisson kernel $P_t$, we refer to \cite[Chapter 7]{Lin} for a detailed discussion.
\end{remark}
\begin{remark}\label{rem:weakbounds}
 Theorem \ref{thm:weak11} and Theorem \ref{thm:limit}  provide, respectively, upper and lower estimates for the weak (1,1) constant of $g_\alpha$ in terms of the dimension $n$, that is
$$\frac{c_{n,\alpha} \omega_{n-1}}{n}\sqrt{\frac{n!}{2\prod_{k=1}^{n-1}(k+\alpha)}}\lesssim ||g_\alpha||_{w(1,1)}\lesssim C_\alpha n^2.$$

%Specifically, Theorem \ref{thm:weak11} provides the uniform upper bound that grows at most like $c_{n,\alpha} \omega_{n-1} n^{5/2}$. On the other hand, Corollary \ref{c1} shows the lower bound is at least $c_{n,\alpha} \omega_{n-1} \sqrt{\frac{n!}{2\prod_{k=1}^{n-1}(k+\alpha)}}$.
\end{remark}

The organization of the paper is as follows. In Section 2, we establish dimension-free \(L^p\) bounds for \(g_\alpha\) for every \(1<p<\infty\) and identify its exact \(L^2\) norm as \(\sqrt{\alpha/2}\). In Section 3, we prove Theorem\ref{thm:weak11}  via the Calder\'on-Zygmund decomposition and detailed kernel estimates. In Section 4, we prove Theorem \ref{thm:limit}  by a scaling argument combined with sharp pointwise asymptotics of the gradient kernel. Finally, we give the computation of the derivative of the modified Bessel function of the second kind, $K_\nu(z)$ , in the appendix.

Throughout the paper,  $A \lesssim B$ means $A \leq C B$ for some constant $C>0$, and $|E|$ denotes the Lebesgue measure of a measurable set $E\subset \mathbb R^n$.

\section{Dimension-free estimates on $L^p(\mathbb R^n)$ }
In this section, we show that the $L^p(1<p<\infty)$  bound of $g_\alpha$ is independent of the dimension $n$ and  we obtain the $L^2$ identities with explicit constant.

%The proof of Theorem~\ref{thm:main} is completely explicit once \eqref{eq:stein-classical} is granted. The key identity is
%\begin{equation}\label{eq:intro-mixture}
%P_t^\alpha(x)
%=\int_1^\infty P_{rt}^1(x)\,d\mu_\alpha(r),
%\qquad 0<\alpha<1,
%\end{equation}
%where
%\begin{equation}\label{eq:intro-mu}
%d\mu_\alpha(r)
%=\frac{2\sqrt\pi}{\Gamma(\alpha/2)\Gamma((1-\alpha)/2)}
%(r^2-1)^{-(1+\alpha)/2}\,d r,
%\qquad r>1.
%\end{equation}
%We prove directly that $\mu_\alpha$ is a probability measure and verify \eqref{eq:intro-mixture} by a beta-integral computation.

We write the classical Poisson kernel as
\begin{equation}\label{eq:classical-poisson}
\mathbb P_s(x):=P_s^1(x)
=c_{n,1}\frac{s}{(s^2+|x|^2)^{(n+1)/2}},
\qquad
c_{n,1}=\frac{\Gamma((n+1)/2)}{\pi^{(n+1)/2}}.
\end{equation}
For $0<\alpha<1$, define
\begin{equation}\label{eq:kappa-alpha}
\kappa_\alpha
:=\frac{2\sqrt\pi}{\Gamma(\alpha/2)\Gamma((1-\alpha)/2)}
\end{equation}
and
\begin{equation}\label{eq:mu-alpha}
d\mu_\alpha(r)
:=\kappa_\alpha(r^2-1)^{-(1+\alpha)/2}\,d r,
\qquad r>1.
\end{equation}

\begin{lemma}\label{lem:probability}
For every $0<\alpha<1$, the measure $\mu_\alpha$ defined by \eqref{eq:mu-alpha} is a probability measure on $(1,\infty)$.
\end{lemma}

\begin{proof}
Set $s=r^{-2}$. Then $r=s^{-1/2}$ and
\[
d r=-\frac12s^{-3/2}\,d s.
\]
Moreover,
\[
r^2-1=\frac{1-s}{s}.
\]
Therefore
\begin{align*}
\int_1^\infty (r^2-1)^{-(1+\alpha)/2}\,d r
&=\frac12\int_0^1
s^{\alpha/2-1}(1-s)^{(1-\alpha)/2-1}\,d s\\
&=\frac12 B\!\left(\frac\alpha2,\frac{1-\alpha}{2}\right)\\
&=\frac12\frac{\Gamma(\alpha/2)\Gamma((1-\alpha)/2)}{\Gamma(1/2)}.
\end{align*}
Since $\Gamma(1/2)=\sqrt\pi$, multiplication by \eqref{eq:kappa-alpha} gives
\[
\int_1^\infty d \mu_\alpha(r)=1.
\]
\end{proof}

\begin{lemma}[Positive Poisson representation]\label{lem:mixture}
Let $0<\alpha<1$. Then, for every $x\in \mathbb R^n$ and $t>0$,
\begin{equation}\label{eq:mixture-kernel}
P_t^\alpha(x)=\int_1^\infty \mathbb P_{rt}(x)\,d\mu_\alpha(r).
\end{equation}
Consequently, for every $f\in \mathscr{S}(\mathbb R^n)$,
\begin{equation}\label{eq:mixture-function}
P_t^\alpha*f(x)
=\int_1^\infty \mathbb P_{rt}*f(x)\,d\mu_\alpha(r).
\end{equation}
\end{lemma}

\begin{proof}
Let $\rho=|x|/t$. By \eqref{eq:classical-poisson} and \eqref{eq:mu-alpha}, we obtain that
\begin{align*}
\int_1^\infty \mathbb P_{rt}(x)\,d\mu_\alpha(r)
&=\kappa_\alpha c_{n,1}
\int_1^\infty
\frac{rt}{(r^2t^2+|x|^2)^{(n+1)/2}}
(r^2-1)^{-(1+\alpha)/2}\,d r\notag\\
&=\kappa_\alpha c_{n,1}t^{-n}
\int_1^\infty
\frac{r(r^2-1)^{-(1+\alpha)/2}}
{(r^2+\rho^2)^{(n+1)/2}}\,d r.
\label{eq:mixture-proof1}
\end{align*}
Put $y=r^2-1$, so that $r\,d r=\frac12d y$. Then we have
\begin{equation*}\label{eq:mixture-proof2}
\int_1^\infty \mathbb P_{rt}(x)\,d\mu_\alpha(r)=\frac{\kappa_\alpha c_{n,1}}2t^{-n}
\int_0^\infty
y^{-(1+\alpha)/2}
(y+1+\rho^2)^{-(n+1)/2}\,d y.
\end{equation*}
Set
\[
a=\frac{1-\alpha}{2},\qquad b=\frac{n+1}{2},\qquad A=1+\rho^2.
\]
Since $0<a<b$, the standard beta-integral identity
\begin{equation*}\label{eq:beta-integral}
\int_0^\infty y^{a-1}(y+A)^{-b}\,d y
=A^{a-b}B(a,b-a)
\end{equation*}
gives
\begin{align*}
\int_1^\infty \mathbb P_{rt}(x)\,d\mu_\alpha(r)=\frac{\kappa_\alpha c_{n,1}}2t^{-n}
(1+\rho^2)^{-(n+\alpha)/2}
\frac{\Gamma((1-\alpha)/2)\Gamma((n+\alpha)/2)}{\Gamma((n+1)/2)}.
\label{eq:mixture-proof3}
\end{align*}
Using the definitions of $\kappa_\alpha$ and $c_{n,1}$,
\begin{align*}
\frac{\kappa_\alpha c_{n,1}}2
\frac{\Gamma((1-\alpha)/2)\Gamma((n+\alpha)/2)}{\Gamma((n+1)/2)}
&=\frac{\Gamma((n+\alpha)/2)}{\pi^{n/2}\Gamma(\alpha/2)}\\
&=c_{n,\alpha}.
\end{align*}
Finally,
\[
t^{-n}\left(1+\frac{|x|^2}{t^2}\right)^{-(n+\alpha)/2}
=\frac{t^\alpha}{(t^2+|x|^2)^{(n+\alpha)/2}}.
\]
This proves \eqref{eq:mixture-kernel}. Formula \eqref{eq:mixture-function} follows by Fubini's theorem; for $f\in \mathscr{S}(\mathbb R^n)$ all integrals are absolutely convergent.
\end{proof}

\begin{lemma}[Point-wise estimate]\label{lemm:pointwise}
For every $0<\alpha<1$, $f\in \mathscr{S}(\mathbb{R}^n)$, and $x\in \mathbb{R}^n$,
\begin{equation}\label{eq:pointwise}
g_\alpha(f)(x)\le g(f)(x),
\end{equation}
whenever $g(f)(x)<\infty$. For $\alpha=1$, equality holds.
\end{lemma}
\begin{proof}
 The argument is written first for truncated square functions so that all applications of Minkowski's inequality are immediate.

For $0<\varepsilon<R<\infty$, define
\begin{equation*}\label{eq:truncated-g}
g_{\alpha;\varepsilon,R}(f)(x)
:=\left(\int_\varepsilon^R
|\nabla_{x,t}u_\alpha(x,t)|^2\,t\,d t\right)^{1/2}.
\end{equation*}
Let
\begin{equation*}\label{eq:v-def}
v(x,s):=\mathbb P_s*f(x).
\end{equation*}
By Lemma~\ref{lem:mixture}, for $0<\alpha<1$,
\begin{equation*}\label{eq:u-average-v}
u_\alpha(x,t)=\int_1^\infty v(x,rt)\,d\mu_\alpha(r).
\end{equation*}
For the Schwartz class of functions, differentiation under the integral is justified and gives
\begin{equation}\label{eq:derivatives-average}
\nabla_xu_\alpha(x,t)
=\int_1^\infty \nabla_xv(x,rt)\,d\mu_\alpha(r),
\qquad
\partial_tu_\alpha(x,t)
=\int_1^\infty r\,\partial_sv(x,rt)\,d\mu_\alpha(r).
\end{equation}

Using $t\,d t=\frac{d t}{t}\,t^2$, we may write
\begin{equation*}\label{eq:g-hilbert}
g_{\alpha;\varepsilon,R}(f)(x)
=\left(\int_\varepsilon^R
|t\nabla_{x,t}u_\alpha(x,t)|^2\frac{d t}{t}\right)^{1/2}.
\end{equation*}
By \eqref{eq:derivatives-average} and Minkowski's integral inequality in the Hilbert space
\[
L^2\!\left((\varepsilon,R),\frac{d t}{t};\mathbb{C}^{n+1}\right),
\]
we obtain
\begin{align}
g_{\alpha;\varepsilon,R}(f)(x)
&\le\int_1^\infty
\Bigg[\int_\varepsilon^R
\left(
|t\nabla_xv(x,rt)|^2
+|rt\,\partial_sv(x,rt)|^2
\right)\frac{d t}{t}\Bigg]^{1/2}
d\mu_\alpha(r).
\label{eq:minkowski}
\end{align}
Fix $r\ge1$ and set $s=rt$. Then $d t/t=d s/s$ and $t=s/r$. Hence the square of the bracketed quantity in \eqref{eq:minkowski} is
\begin{equation*}\label{eq:scaled-bracket}
\int_{r\varepsilon}^{rR}
\left(
\frac1{r^2}|s\nabla_xv(x,s)|^2
+|s\partial_sv(x,s)|^2
\right)\frac{d s}{s}.
\end{equation*}
Because $r\ge1$, we have $r^{-2}\le1$. Therefore

	\begin{align}\label{eq:bracket-by-g1}
	&\int_{r\varepsilon}^{rR}
\left(
\frac1{r^2}|s\nabla_xv(x,s)|^2
+|s\partial_sv(x,s)|^2
\right)\frac{d s}{s}\\
&\le\int_0^\infty
\left(
|s\nabla_xv(x,s)|^2
+|s\partial_sv(x,s)|^2
\right)\frac{d s}{s}\notag\\
&=\int_0^\infty |\nabla_{x,s}v(x,s)|^2\,s\,d s
=g(f)(x)^2.\notag
	\end{align}

%\begin{align}\label{eq:bracket-by-g1}
%&\int_{r\varepsilon}^{rR}
%\left(
%\frac1{r^2}|s\nabla_xv(x,s)|^2
%+|s\partial_sv(x,s)|^2
%\right)\frac{d s}{s}\\
%&\le\int_0^\infty
%\left(
%|s\nabla_xv(x,s)|^2
%+|s\partial_sv(x,s)|^2
%\right)\frac{d s}{s}\notag\\
%&=\int_0^\infty |\nabla_{x,s}v(x,s)|^2\,s\,d s
%=g(f)(x)^2.
%\end{align}
Combining \eqref{eq:minkowski} and \eqref{eq:bracket-by-g1}, and using Lemma~\ref{lem:probability},
\[
g_{\alpha;\varepsilon,R}(f)(x)
\le\int_1^\infty g(f)(x)\,d\mu_\alpha(r)
=g(f)(x).
\]
Letting $\varepsilon\downarrow0$ and $R\uparrow\infty$, monotone convergence applied to the nonnegative square integrals yields \eqref{eq:pointwise}. If $\alpha=1$, then $u_\alpha=v$ and there is equality by definition.

\end{proof}

\begin{proof}[Proof of Theorem~\ref{thm:main}]

For every \(0<\alpha<1\) and \(f\in \mathscr{S}(\mathbb{R}^n)\), Lemma~\ref{lemm:pointwise} yields the point-wise estimate

$$
g_\alpha(f)(x)\leq g(f)(x).
$$

 Hence, by the classical dimension-free estimate \eqref{eq:stein-classical},

$$
\|g_\alpha(f)\|_{L^p(\mathbb{R}^n)}
\leq \|g(f)\|_{L^p(\mathbb{R}^n)}
\leq C_p\|f\|_{L^p(\mathbb{R}^n)},
$$

where \(C_p\) is independent of both $n$ and $\alpha$. The estimate extends to all \(f\in L^p(\mathbb{R}^n)\) by density.

\end{proof}

To prove Theorem \ref{thm:L2}, we will make use of standard formulas for Bessel functions, which can be found in the reference book by Gradshteyn and Ryzhik \cite{GR}.
\begin{lemma}\label{lemma1}\cite[p.368 3.471.9]{GR}
Let $\nu$ be an arbitrary complex number, and let $\beta$ and $\gamma$ be complex parameters such that  \textbf{\textsf{Re}} \,$\beta>0,$ \textbf{\textsf{Re}}\, $\gamma>0$. Then
\[
\int_0^\infty s^{\nu-1} e^{-\beta s - \gamma/s}\,ds
= 2\left(\frac{\gamma}{\beta}\right)^{\nu/2} K_\nu(2\sqrt{\beta\gamma}),
\qquad
\]
where $K_\nu(z)= \frac{\pi i}{2}\,e^{\frac{\pi}{2}\nu i}\; H_\nu^{(1)}\!\bigl(z e^{\frac{\pi}{2}i}\bigr)$ denotes the modified Bessel function of the second kind(see Lemma \ref{derivative} in the appendix for the definition).

\end{lemma}
\begin{lemma}\label{lemma2} \cite[p.684 6.576.4]{GR}
Let $\lambda$ and $\nu$ be complex parameters such that $\textbf{\textsf{Re}} \, \lambda > 2|\textbf{\textsf{Re}}  \,\nu|$. Then
\[
\int_0^\infty s^{\lambda-1} K_\nu(s)^2\,ds
= \frac{2^{\lambda-3}}{\Gamma(\lambda)}
\Gamma\left(\frac{\lambda}{2}+\nu\right)
\Gamma\left(\frac{\lambda}{2}-\nu\right)
\Gamma\left(\frac{\lambda}{2}\right)^2.
\]

\end{lemma}

We now give the proof of Theorem  \ref{thm:L2}.
\begin{proof}[Proof of Theorem~\ref{thm:L2}]

We begin by computing the Fourier transform of the fractional Poisson kernel. Using the standard Gamma function integral representation
\[
(t^2 + |x|^2)^{-\frac{n+\alpha}{2}}
= \frac{1}{\Gamma(\frac{n+\alpha}{2})}
\int_0^\infty s^{\frac{n+\alpha}{2}-1} e^{-s(t^2+|x|^2)}\,ds,
\]
together with the Fourier transform of the Gaussian function
\[
\int_{\mathbb R^n} e^{-s|x|^2} e^{-2\pi i x\cdot \xi}\,dx
= \left(\frac{\pi}{s}\right)^{n/2} e^{-\pi^2 |\xi|^2/s},
\]
we obtain, for $\xi\in \mathbb R^n$,
\[
\widehat{P_t^\alpha}(\xi)
= \int_{\mathbb R^n} P_t^\alpha(x) e^{-2\pi i x\cdot \xi}\,dx
= \frac{t^\alpha}{\Gamma(\alpha/2)}
\int_0^\infty s^{\alpha/2 - 1} e^{-t^2 s - \pi^2 |\xi|^2/s}\,ds.
\]

Applying Lemma \ref{lemma1} with $\nu = \alpha/2$, $\beta = t^2$, and $\gamma = \pi^2|\xi|^2$, we get
\[
\widehat{P_t^\alpha}(\xi)
= \frac{2^{1-\alpha/2}}{\Gamma(\alpha/2)} (2\pi t|\xi|)^{\alpha/2} K_{\alpha/2}(2\pi t|\xi|).
\]
%where $K_\nu(z)= \frac{\pi i}{2}\,e^{\frac{\pi}{2}\nu i}\; H_\nu^{(1)}\!\bigl(z e^{\frac{\pi}{2}i}\bigr)$ denotes the modified Bessel function of the second kind(see Lemma \ref{derivative} in the appendix for the definition).

Introducing the multiplier
\[
m_\alpha(s) := \frac{2^{1-\alpha/2}}{\Gamma(\alpha/2)} s^{\alpha/2} K_{\alpha/2}(s).
\]
and setting $s = 2\pi t|\xi|,$ we find that
$$\widehat{P_t^\alpha}(\xi) = m_\alpha(2\pi t|\xi|).$$

Since $u = P_t^\alpha * f$, the Fourier transform of the extension is given by
\[
\widehat{u}(\xi,t) = m_\alpha(2\pi t|\xi|) \widehat{f}(\xi).
\]

Differentiating with respect to the spatial and temporal variables yields
\[
\widehat{\partial_{x_j} u}(\xi,t) = 2\pi i \xi_j \, m_\alpha(2\pi t|\xi|) \widehat{f}(\xi),
\]
and
\[
\widehat{\partial_t u}(\xi,t) = 2\pi |\xi| \, m_\alpha'(2\pi t|\xi|) \widehat{f}(\xi).
\]

By Plancherel's theorem, we have
\[
\int_{\mathbb R^n} |\nabla u(x,t)|^2\,dx
= \int_{\mathbb R^n} (2\pi |\xi|)^2
\left( |m_\alpha(2\pi t|\xi|)|^2 + |m_\alpha'(2\pi t|\xi|)|^2 \right)
|\widehat{f}(\xi)|^2\,d\xi.
\]

Therefore
\begin{align*}
\|g_\alpha(f)\|_{L^2(\mathbb R^n)}^2
= &\int_0^\infty \int_{\mathbb R^n} |\nabla u(x,t)|^2\,dx \, t\,dt\\
= &\int_{\mathbb R^n} |\widehat{f}(\xi)|^2
\left( \int_0^\infty (2\pi|\xi|)^2
\left( |m_\alpha(2\pi t|\xi|)|^2 + |m_\alpha'(2\pi t|\xi|)|^2 \right) t\,dt \right) d\xi.
\end{align*}

Making the change of variables $s = 2\pi t|\xi|$, so that $t\,dt = s\,ds/(2\pi|\xi|)^2$, we obtain

\begin{align*}
\|g_\alpha(f)\|_{L^2(\mathbb R^n)}^2
= &\left( \int_0^\infty s \left( |m_\alpha(s)|^2 + |m_\alpha'(s)|^2 \right) ds \right) \|f\|_2^2\\
:=& A_\alpha \|f\|_2^2,\\
\end{align*}
where
\[
A_\alpha = \int_0^\infty s \left( |m_\alpha(s)|^2 + |m_\alpha'(s)|^2 \right) ds.
\]

Next, we show that $A_\alpha =\frac{\alpha}{2},$ which is independent of both $n$ and $\xi$.

Using the recurrence relations for modified Bessel functions(see Lemma \ref{derivative} in the appendix for the proof)
\[
K_\nu'(s) = -\frac{\nu}{s} K_\nu(s) - K_{\nu-1}(s),
\qquad
(s^\nu K_\nu(s))' = -s^\nu K_{\nu-1}(s),
\]
with $\nu = \alpha/2$, we obtain
\[
m_\alpha'(s) = -\frac{2^{1-\alpha/2}}{\Gamma(\alpha/2)} s^{\alpha/2} K_{\alpha/2 - 1}(s).
\]

Consequently,
\[
s\left( |m_\alpha(s)|^2 + |m_\alpha'(s)|^2 \right)
= \left(\frac{2^{1-\alpha/2}}{\Gamma(\alpha/2)}\right)^2
s^{\alpha+1}
\left( K_{\alpha/2}(s)^2 + K_{\alpha/2 - 1}(s)^2 \right).
\]

It follows that
\[
A_\alpha = \left(\frac{2^{1-\alpha/2}}{\Gamma(\alpha/2)}\right)^2 (I_1 + I_2),
\]
where
\[
I_1 := \int_0^\infty s^{\alpha+1} K_{\alpha/2}(s)^2\,ds,
\qquad
I_2 := \int_0^\infty s^{\alpha+1} K_{\alpha/2 - 1}(s)^2\,ds.
\]

%We now turn to estimating $I_1$ and $I_2$ individually.
%We invoke the standard integral identity (see \cite[Appendix]{Stein71})
%\[
%\int_0^\infty s^{\lambda-1} K_\mu(s)^2\,ds
%= \frac{2^{\lambda-3}}{\Gamma(\lambda)}
%\Gamma\left(\frac{\lambda}{2}+\mu\right)
%\Gamma\left(\frac{\lambda}{2}-\mu\right)
%\Gamma\left(\frac{\lambda}{2}\right)^2,
%\]
%valid for $\lambda > 2|\mu|$.
By Lemma \ref{lemma2}, with $\lambda = \alpha+2$, we apply the formula with $\mu = \alpha/2$ for $I_1$, and $\mu = \alpha/2 - 1$ for $I_2$.  This yields
\[
I_1 = \frac{2^{\alpha-1}}{\Gamma(\alpha+2)}
\Gamma(\alpha+1)\Gamma(1)\Gamma\left(\frac{\alpha}{2}+1\right)^2,
\]
and
\[
I_2 = \frac{2^{\alpha-1}}{\Gamma(\alpha+2)}
\Gamma(\alpha)\Gamma(2)\Gamma\left(\frac{\alpha}{2}+1\right)^2.
\]

Using $\Gamma(1)=\Gamma(2)=1$, $\Gamma(\alpha+1)=\alpha\Gamma(\alpha)$, and $\Gamma(\alpha+2)=(\alpha+1)\alpha\Gamma(\alpha)$, we get
\[
I_1 + I_2
= \frac{2^{\alpha-1}\Gamma(\alpha/2+1)^2}{\Gamma(\alpha+2)}
\left( \Gamma(\alpha) + \alpha\Gamma(\alpha) \right)
= \frac{2^{\alpha-1}\Gamma(\alpha/2+1)^2}{\alpha}=2^{\alpha-3} \alpha \Gamma(\alpha/2)^2.
\]

%Since $\Gamma(\alpha/2+1) = \frac{\alpha}{2}\Gamma(\alpha/2)$, it follows that
%\[
%I_1 + I_2 = 2^{\alpha-3} \alpha \Gamma(\alpha/2)^2.
%\]

Therefore
\[
A_\alpha
= \left(\frac{2^{1-\alpha/2}}{\Gamma(\alpha/2)}\right)^2
\cdot 2^{\alpha-3} \alpha \Gamma(\alpha/2)^2
= \frac{\alpha}{2}.
\]

This proves
\[
\|g_\alpha(f)\|_{L^2(\mathbb R^n)}^2 = \frac{\alpha}{2}\|f\|_2^2,
\]
which is the desired $L^2$-identity.

\end{proof}

\section{Proof of Theorem \ref{thm:weak11}}

The proof of Theorem \ref{thm:weak11} relies on a modified Calder\'on-Zygmund decomposition due to Janakiraman \cite[Lemmas 5.1-5.2]{Janakiraman}. For the convenience of the reader, we recall this decomposition in the following lemma.

\begin{lemma}\cite[Lemmas 5.1-5.2]{Janakiraman}\label{lem:CZ}
	Let $f\in L^1(\mathbb{R}^n)$ and $\lambda>0$. Then there exists a decomposition $\mathbb{R}^n = G \cup E$ with $G\cap E = \varnothing$, where $E = \bigcup_k Q_k$ is a union of semi-cubes of the form $Q_k = \prod_{i=1}^n [a_i,b_i]$, with side lengths either $a$ or $2a$ for some $a>0$, and with pairwise disjoint interiors. Moreover, the following properties hold:
%	\begin{enumerate}

		(i) $\displaystyle \lambda \leq \frac{1}{|Q_k|}\int_{Q_k} |f(y)|\,dy \leq 2\lambda$ for each $k$;

		(ii) $\displaystyle |E| \leq \frac{1}{\lambda}\|f\|_{L^1(\mathbb{R}^n)}$;

		(iii) $|f(x)| \leq \lambda$ for almost every $x\in G$;

		(iv) $f = g + b$;
		
		(v) $\|g\|_{L^2(\mathbb{R}^n)}^2 \leq 5\lambda \|f\|_{L^1(\mathbb{R}^n)}$;

		(vi) $b = \sum_k b_k$, with each $b_k$ supported in $Q_k$, satisfying $\int b_k(x)\,dx = 0$, and

		\; \quad  \quad $\sum_k \|b_k\|_{L^1(\mathbb{R}^n)} \leq 2\|f\|_{L^1(\mathbb{R}^n)}.$
	%\end{enumerate}
\end{lemma}

The proof of Theorem \ref{thm:weak11} is primarily based on the approach used in \cite[Theorem 1.1]{Lai}, and we generalize this idea to the setting of fractional square functions.  In the new setting, the problem becomes rather complicated, and the Gamma function and Beta function will play an important role.
\begin{proof}
Fix $\lambda>0$ and $f\in L^1(\mathbb{R}^n)$. Applying Lemma \ref{lem:CZ} to $f$ at height $\lambda$, we obtain the decomposition $f = g + b$.
Since $g_\alpha$ is subadditive, we have
\begin{equation}\label{eq:subadd}
	\begin{aligned}
		&\left|\left\{x\in \mathbb{R}^n : g_\alpha(f)(x) > \lambda \right\}\right| \\
		&\qquad \leq
		\left|\left\{x\in \mathbb{R}^n : g_\alpha(g)(x) > \lambda/2 \right\}\right|
		+
		\left|\left\{x\in \mathbb{R}^n : g_\alpha(b)(x) > \lambda/2 \right\}\right|.
	\end{aligned}
\end{equation}

We first estimate the contribution of the good part $g$. By Chebyshev's inequality, Lemma \ref{lem:CZ}(v), and the $L^2$-boundedness of $g_\alpha$ established in Theorem \ref{thm:L2}, we obtain
\begin{equation}\label{eq:good}
	\begin{aligned}
		\left|\left\{x\in \mathbb{R}^n : g_\alpha(g)(x) > \lambda/2 \right\}\right|
		&\leq \frac{4}{\lambda^2} \|g_\alpha(g)\|_{L^2(\mathbb{R}^n)}^2 \\
		&\leq \frac{2\alpha}{\lambda^2} \|g\|_{L^2(\mathbb{R}^n)}^2 \\
		&\leq \frac{10\alpha}{\lambda} \|f\|_{L^1(\mathbb{R}^n)}.
	\end{aligned}
\end{equation}

It remains to control the bad part $b$. For each cube $Q_k$, let $y_k$ denote its center, and let its side length be either $a_k$ or $2a_k$. Then we have the following geometric estimates:
\begin{equation}\label{eq:dist}
	\frac{a_k}{2} \leq \inf_{y\in \partial Q_k} |y - y_k|
	\leq \sup_{y\in \partial Q_k} |y - y_k|
	\leq a_k \sqrt{n}.
\end{equation}

Let $\bar d_k$ denote the diameter of $Q_k$. Then
\[
\sqrt{n} a_k \leq \bar d_k \leq 2a_k \sqrt{n}.
\]

Define the enlarged set
\begin{equation}\label{eq:Qstar}
	Q_k^* := \bigcup_{y\in Q_k} B\left(y, \frac{\bar d_k}{4n^{3/2}}\right).
\end{equation}

Clearly, $Q_k \subset Q_k^*$. Moreover, we claim that $Q_k^* \subset (1+1/n)Q_k$, where $(1+1/n)Q_k$ denotes the concentric semi-cube with side length multiplied by $1+1/n$. Indeed, for any $x\in Q_k^*$,
\[
\operatorname{dist}(x,Q_k) \leq \frac{\bar d_k}{4n^{3/2}} \leq \frac{a_k}{2n},
\]
which implies $x\in (1+1/n)Q_k$.

Set
\[
H := \bigcup_k Q_k^*, \qquad F := H^c.
\]
Then
\begin{equation}\label{eq:splitH}
	\left|\left\{x\in \mathbb{R}^n : g_\alpha(b)(x) > \lambda/2 \right\}\right|
	\leq |H| + \left|\left\{x\in F : g_\alpha(b)(x) > \lambda/2 \right\}\right|.
\end{equation}

Applying Lemma \ref{lem:CZ}(ii), we obtain the estimate
\begin{equation}\label{eq:H}
	\begin{aligned}
		|H|\leq \sum_k |Q_k^*|
		\leq \left(1+\frac{1}{n}\right)^n \sum_k |Q_k| \leq e |E|
		\leq \frac{e}{\lambda} \|f\|_{L^1(\mathbb{R}^n)}.
	\end{aligned}
\end{equation}

Therefore, it suffices to establish the following estimate
\begin{equation}\label{eq:main}
	\left|\left\{x\in F : g_\alpha(b)(x) > \lambda/2 \right\}\right|
	\leq \frac{C_\alpha n^2}{\lambda} \|f\|_{L^1(\mathbb{R}^n)}.
\end{equation}

By Chebyshev's inequality and Lemma \ref{lem:CZ}(vi), we have
\begin{equation}\label{eq:badsum}
	\begin{aligned}
		\left|\left\{x\in F : g_\alpha(b)(x) > \lambda/2 \right\}\right|
		&\leq \frac{2}{\lambda} \sum_k \int_F g_\alpha(b_k)(x)\,dx \\
		&\leq \frac{2}{\lambda} \sum_k \int_{(Q_k^*)^c} g_\alpha(b_k)(x)\,dx.
	\end{aligned}
\end{equation}

It suffices to consider the case $k=1$, since all other terms are treated in the same manner. Note that $b_1$ is supported in $Q_1$ and has mean value zero. Without loss of generality, we may assume that $Q_1$ is the semi-cube with center $y_0$ and side length $a$ (or $2a$) for some $a>0$. Recall that
\[
g_\alpha(b_1)(x)
= \left( \int_0^\infty |\nabla (P_t^\alpha * b_1)(x)|^2 t\,dt \right)^{1/2},
\]
where
\[
|\nabla (P_t^\alpha * b_1)(x)|^2
= \sum_{i=1}^n |\partial_{x_i} P_t^\alpha * b_1(x)|^2
+ |\partial_t P_t^\alpha * b_1(x)|^2.
\]

Denote $u(x,t) := P_t^\alpha * b_1(x)$. We use the cancellation of $b_1$ in the following way:
\begin{equation}\label{eq:diff}
	u(x,t)
	= \int_{Q_1} b_1(y) \big( P_t^\alpha(x-y) - P_t^\alpha(x-y_0) \big)\,dy.
\end{equation}

A direct differentiation yields the following decompositions. For the spatial derivatives, we define
\begin{equation}\label{eq:A}
	A_i(x,t) :=
	 c_{n,\alpha} (n+\alpha)
	\int_{Q_1} b_1(y) t^\alpha
	\frac{y_i - y_{0,i}}{(t^2 + |x-y|^2)^{\frac{n+\alpha+2}{2}}}\,dy
\end{equation}
and
\begin{equation}\label{eq:B}
	\begin{aligned}
		B_i(x,t) := &- c_{n,\alpha} (n+\alpha)
		\int_{Q_1} b_1(y) t^\alpha \\
		&\times
		\left[
		\frac{x_i - y_{0,i}}{(t^2 + |x-y|^2)^{\frac{n+\alpha+2}{2}}}
		-
		\frac{x_i - y_{0,i}}{(t^2 + |x-y_0|^2)^{\frac{n+\alpha+2}{2}}}
		\right] dy,
	\end{aligned}
\end{equation}
such that
\[
\partial_{x_i} u(x,t) = A_i(x,t) + B_i(x,t).
\]

For the time derivative, we define
\begin{equation}\label{eq:D}
	\begin{aligned}
		D(x,t) := &c_{n,\alpha} \alpha
		\int_{Q_1} b_1(y) \\
		&\times
		\left[
		\frac{t^{\alpha-1}}{(t^2 + |x-y|^2)^{\frac{n+\alpha}{2}}}
		-
		\frac{t^{\alpha-1}}{(t^2 + |x-y_0|^2)^{\frac{n+\alpha}{2}}}
		\right] dy,
	\end{aligned}
\end{equation}
and
\begin{equation}\label{eq:E}
	\begin{aligned}
		E(x,t) := &c_{n,\alpha} (n+\alpha)
		\int_{Q_1} b_1(y) \\
		&\times
		\left[
		\frac{t^{\alpha+1}}{(t^2 + |x-y|^2)^{\frac{n+\alpha+2}{2}}}
		-
		\frac{t^{\alpha+1}}{(t^2 + |x-y_0|^2)^{\frac{n+\alpha+2}{2}}}
		\right] dy,
	\end{aligned}
\end{equation}
such that
\[
\partial_t u(x,t) = D(x,t) - E(x,t).
\]

From now on,  we will suppress the arguments of $A_i(x,t), B_i(x,t), D(x,t), E(x,t)$, etc., and denote them simply by $A_i, B_i, D, E.$
In the following subsections, we estimate the contributions of $A_i$, $B_i$, $D$, and $E$, respectively. To this end, we give some useful inequalities below:
\begin{equation}\label{eq:3.14}
	\begin{aligned}
		\left(\int_{0}^{\infty} \frac{t^{2\alpha+1}}{\left(t^{2}+r^{2}\right)^{n+\alpha+2}}dt\right)^{1/2}&=\left( \int_{0}^{\infty} \frac{x^{2\alpha+1}}{(x^2+1)^{n+\alpha+2}} \, dx\right)^{1/2}r^{-n-1}\\
		&=\left(\frac{1}{2}B(\alpha+1,n+1)\right)^{1/2} r^{-n-1}\\
&\leq \left( \frac{1}{2} \Gamma (\alpha + 1) n^{-\left( \alpha + 1 \right)} \right)^{\frac{1}{2}}r^{-n-1}\\
		&= C_\alpha n^{-(\alpha+1)/2}r^{-n-1}.
	\end{aligned}
	\end{equation}

Similarly, we have
	\begin{equation}\label{eq:3.15}
		\begin{aligned}
			\left(\int_{0}^{\infty} \frac{t^{2\alpha+1}}{\left(t^{2}+r^{2}\right)^{n+\alpha+4}}dt\right)^{1/2}
			&=\left(\frac{1}{2}B(\alpha+1,n+3)\right)^{1/2} r^{-n-3}\\
			&\leq C_\alpha n^{-(\alpha+1)/2}r^{-n-3}.
		\end{aligned}
		\end{equation}
		\begin{equation}\label{eq:3.16}
			\begin{aligned}
				\left(\int_{0}^{\infty} \frac{t^{2\alpha-1}}{\left(t^{2}+r^{2}\right)^{n+\alpha+2}}dt\right)^{1/2}
				&=\left(\frac{1}{2}B(\alpha,n+2)\right)^{1/2} r^{-n-2}\\
				&\leq C_\alpha n^{-\alpha/2}r^{-n-2}.
			\end{aligned}
		\end{equation}
		\begin{equation}\label{eq:3.17}
			\begin{aligned}
				\left(\int_{0}^{\infty} \frac{t^{2\alpha-1}}{\left(t^{2}+r^{2}\right)^{n+\alpha}}dt\right)^{1/2}
				&=\left(\frac{1}{2}B(\alpha,n)\right)^{1/2} r^{-n}\\
				&\leq C_\alpha n^{-\alpha/2}r^{-n}.
			\end{aligned}
			\end{equation}
			\begin{equation}\label{eq:3.18}
				\begin{aligned}
					\left(\int_{0}^{\infty} \frac{t^{2\alpha+3}}{\left(t^{2}+r^{2}\right)^{n+\alpha+4}}dt\right)^{1/2}
					&=\left(\frac{1}{2}B(\alpha+2,n+2)\right)^{1/2} r^{-n-2}\\
					&\leq C_\alpha n^{-(\alpha+2)/2}r^{-n-2}.
				\end{aligned}
			\end{equation}
			\begin{equation}\label{eq:3.19}
				\begin{aligned}
					\left(\int_{0}^{\infty} \frac{t^{2\alpha+3}}{\left(t^{2}+r^{2}\right)^{n+\alpha+2}}dt\right)^{1/2}
					&=\left(\frac{1}{2}B(\alpha+2,n)\right)^{1/2} r^{-n}\\
					&\leq C_\alpha n^{-(\alpha+2)/2}r^{-n}.
				\end{aligned}
			\end{equation}

\subsection{Estimate of $A_i$}

We split $A_i$ into two parts according to the relative size of $|x-y|$ and $|y-y_0|$:
\[
A_i(x,t) = A_{i,1}(x,t) + A_{i,2}(x,t),
\]
where
\[
A_{i,1}(x,t)
:= \int_{Q_1 \cap \{y : |x-y| > n|y-y_0|\}}
c_{n,\alpha} (n+\alpha) b_1(y) t^\alpha
\frac{y_i - y_{0,i}}{(t^2 + |x-y|^2)^{\frac{n+\alpha+2}{2}}}\,dy,
\]
and
\[
A_{i,2}(x,t)
:= \int_{Q_1 \cap \{y : |x-y| \leq n|y-y_0|\}}
c_{n,\alpha} (n+\alpha) b_1(y) t^\alpha
\frac{y_i - y_{0,i}}{(t^2 + |x-y|^2)^{\frac{n+\alpha+2}{2}}}\,dy.
\]

\medskip

\noindent \textbf{Estimate of $A_{i,1}$.}
For $x\in F$ and $y\in Q_1$ with $|x-y| > n|y-y_0|$, we have
\[
\frac{n}{n+1} |x-y_0| < |x-y| < \frac{n}{n-1} |x-y_0|.
\]

For this part, Minkowski's inequality gives
\[
\begin{aligned}
	\left( \int_{0}^{\infty} \sum_{i=1}^{n} \big|A_{i,1}(x,t)\big|^2 t \,dt \right)^{\frac12}
	&\le c_{n,\alpha}(n+\alpha)
	\int_{\substack{y\in Q_1 \\ |x-y|>n|y-y_0|}}
	|b_1(y)|
	\left( \sum_{i=1}^{n} \big|y_i - y_{0,i}\big|^2 \right)^{\frac12}\\
	&\times
	\left( \int_{0}^{\infty} \frac{t^{2\alpha+1}}{\big(t^2 + |x-y|^2\big)^{n+\alpha+2}} dt \right)^{\frac12} dy.
\end{aligned}
\]
By (\ref{eq:3.14}), we get
\[
\begin{aligned}
	\left(
	\int_0^\infty\sum_{i=1}^n|A_{i,1}(x,t)|^2t\,dt
	\right)^{1/2}
	&\le c_{n,\alpha}(n+\alpha)
	\left(\frac12B(\alpha+1,n+1)\right)^{1/2}\\
	&\times\int_{Q_1}|b_1(y)|\,|y-y_0|\,|x-y|^{-n-1}\,dy.
\end{aligned}
\]

In this region
\[
|x-y_0|\le |x-y|+|y-y_0|
<\left(1+\frac1n\right)|x-y|,
\]
and hence
\[
|x-y|^{-n-1}
\le\left(\frac{n+1}{n}\right)^{n+1}|x-y_0|^{-n-1}
\le C|x-y_0|^{-n-1}.
\]
Using $|y-y_0|\le a\sqrt n$ and $|x-y_0|\ge a/2$, we obtain
\[
\begin{aligned}
	&\int_{F}
	\left(
	\int_0^\infty\sum_{i=1}^n|A_{i,1}(x,t)|^2t\,dt
	\right)^{1/2}dx\\
	&\le Cc_{n,\alpha}(n+\alpha)a\sqrt n
	\left(\frac12B(\alpha+1,n+1)\right)^{1/2}
	\|b_1\|_{L^1}
	\int_{|x-y_0|\ge a/2}|x-y_0|^{-n-1}\,dx\\
	&=2Cc_{n,\alpha}\omega_{n-1}(n+\alpha)\sqrt n
	\left(\frac12B(\alpha+1,n+1)\right)^{1/2}
	\|b_1\|_{L^1(\mathbb{R}^n)} \\
	&\lesssim C_\alpha n\|b_1\|_{L^1(\mathbb{R}^n)},
\end{aligned}
\]
where we used the estimate
$c_{n, \alpha} w_{n-1}\lesssim C_{\alpha} \cdot n^{\frac{\alpha}{2}}$, which can be obtained directly by Stirling's formula.
\medskip

\noindent \textbf{Estimate of $A_{i,2}$.}
For $x\in F$, $y\in Q_1$, and $|x-y| \leq n|y-y_0|$, the definition of $F$ implies
\[
|x-y| \geq \frac{|y-y_0|}{2n^{3/2}}.
\]

Applying Minkowski's inequality and Fubini's theorem, then passing to polar coordinates, we obtain
\[
\begin{aligned}
	&\int_{F}
	\left(
	\int_0^\infty\sum_{i=1}^n|A_{i,2}(x,t)|^2t\,dt
	\right)^{1/2}dx\\
	&\le c_{n,\alpha}(n+\alpha)
	\left(\frac12B(\alpha+1,n+1)\right)^{1/2}
	\int_{Q_1}|b_1(y)|\,|y-y_0|\\
	&\qquad\times
	\int_{\substack{x\in(Q_1^*)^c\\|x-y|\le n|y-y_0|}}
	|x-y|^{-n-1}\,dx\,dy.
\end{aligned}
\]

For each fixed $y\in Q_1$, it follows that
\[
\begin{aligned}
	&|y-y_0|
	\int_{\substack{x\in(Q_1^*)^c\\|x-y|\le n|y-y_0|}}
	|x-y|^{-n-1}\,dx\\
	&\le\omega_{n-1}|y-y_0|
	\int_{|y-y_0|/(2n^{3/2})}^{n|y-y_0|}\frac{ds}{s^2}\\
	&=\omega_{n-1}\left(2n^{3/2}-\frac1n\right)
	\le2\omega_{n-1}n^{3/2}.
\end{aligned}
\]

Consequently
\[
\begin{aligned}
	&\int_{F}
	\left(
	\int_0^\infty\sum_{i=1}^n|A_{i,2}(x,t)|^2t\,dt
	\right)^{1/2}dx\\
	&\le Cc_{n,\alpha}\omega_{n-1}(n+\alpha)n^{3/2}
	\left(\frac12B(\alpha+1,n+1)\right)^{1/2}
	\|b_1\|_{L^1(\mathbb{R}^n)} \\
	&\lesssim C_\alpha n^2\|b_1\|_{L^1(\mathbb{R}^n)} .
\end{aligned}
\]

Thus, we obtain
\[
	\begin{aligned}
		&\int_{F}
		\left(
		\int_0^\infty\sum_{i=1}^n|A_{i,1}(x,t)|^2t\,dt
		\right)^{1/2}dx + \int_{F}
		\left(
		\int_0^\infty\sum_{i=1}^n|A_{i,2}(x,t)|^2t\,dt
		\right)^{1/2}dx\\
		&\lesssim C_\alpha n^2\|b_1\|_{L^1(\mathbb{R}^n)}.
\end{aligned}
\]

\subsection{Estimate of $B_i$}

We decompose $B_i$ analogously:
\[
B_i(x,t) = B_{i,1}(x,t) + B_{i,2}(x,t),
\]
where $B_{i,1}$ corresponds to the region $|x-y| > n|y-y_0|$ and $B_{i,2}$ to $|x-y| \leq n|y-y_0|$.

\medskip

\noindent \textbf{Estimate of $B_{i,1}$.}
In the region $|x-y| > n|y-y_0|$, we have
\[
\frac{n}{n+1}|x-y_0| < |x-y| < \frac{n}{n-1}|x-y_0|,
\]

and
 $$|y-y_0| \leq \sqrt{n} a.$$

  By the mean value theorem, and (\ref{eq:3.15}), we derive
  \[
  \begin{aligned}
  	&\int_{F}
  	\left(
  	\int_0^\infty\sum_{i=1}^n|B_{i,1}(x,t)|^2t\,dt
  	\right)^{1/2}dx\\
  	&\le Cc_{n,\alpha}(n+\alpha)(n+\alpha+2)a\sqrt n
  	\left(\frac12B(\alpha+1,n+3)\right)^{1/2}
  	\|b_1\|_{L^1}\\
  	&\qquad\times
  	\int_{|x-y_0|\ge a/2}|x-y_0|^{-n-1}\,dx\\
  	&\le Cc_{n,\alpha}\omega_{n-1}(n+\alpha)(n+\alpha+2)\sqrt n
  	\left(\frac12B(\alpha+1,n+3)\right)^{1/2}
  	\|b_1\|_{L^1}\\
  	&\lesssim C_\alpha n^2\|b_1\|_{L^1(\mathbb{R}^n)}.
  \end{aligned}
  \]

\medskip

\noindent \textbf{Estimate of $B_{i,2}$.}
In the region $|x-y| \leq n|y-y_0|$, we have
\[
|x-y| \geq \frac{|y-y_0|}{2n^{3/2}},
\qquad
|x-y_0| \geq \frac{|y-y_0|}{2n^{3/2}},
\qquad
|x-y_0| \leq n|y-y_0|.
\]

Minkowski's inequality, Fubini's theorem and (\ref{eq:3.14}) yields
\begin{align*}
	& \int_F \left( \int_0^\infty \sum_{i=1}^n |B_{i,2}(x,t)|^2 t\,dt \right)^{1/2} dx \\
	& \le c_{n,\alpha}(n+\alpha)
	\left(\frac12B(\alpha+1,n+1)\right)^{1/2}
	\int_{Q_1} |b_1(y)|
	\left[
	\int_{\{ |y-y_0|/(2n^{3/2}) \leq |x-y| \leq n|y-y_0| \}}
	\frac{|x-y_0|}{|x-y|^{n+1}} \,dx \right. \\
	& \left.
	\quad +
	\int_{\{ |y-y_0|/(2n^{3/2}) \leq |x-y_0| \leq n|y-y_0| \}}
	\frac{1}{|x-y_0|^n} \,dx
	\right] dy .\\
\end{align*}

Set \(\rho=|y-y_0|,\) and denote
\begin{align*}
	I_1(y)
	&:=
	\int_{\{\rho/(2n^{3/2})\le |x-y|\le n\rho\}}
	\frac{|x-y_0|}{|x-y|^{n+1}}\,dx,\\
	I_2(y)
	&:=
	\int_{\{\rho/(2n^{3/2})\le |x-y_0|
		\le n\rho\}}
	\frac{1}{|x-y_0|^n}\,dx.
\end{align*}

Thus
\begin{align*}
	&\int_F
	\left(
	\int_0^\infty
	\sum_{i=1}^n |B_{i,2}(x,t)|^2\,t\,dt
	\right)^{1/2}dx\\
	&\qquad\le c_{n,\alpha}(n+\alpha)
	\left(\frac12B(\alpha+1,n+1)\right)^{1/2}
	\int_{Q_1}|b_1(y)|
	\bigl(I_1(y)+I_2(y)\bigr)\,dy.
\end{align*}

For \(I_1(y)\), the triangle inequality gives \(|x-y_0|\le |x-y|+|y-y_0|:= |x-y|+\rho.\)  Hence,
\begin{align*}
	I_1(y)
	&\le
	\int_{\{\rho/(2n^{3/2})\le |x-y|\le n\rho\}}
	\frac{|x-y|+\rho}{|x-y|^{n+1}}\,dx\\
	&=
	\omega_{n-1}
	\int_{\rho/(2n^{3/2})}^{n\rho}
	\frac{r+\rho}{r^{n+1}}r^{n-1}\,dr\\
	&=
	\omega_{n-1}
	\int_{\rho/(2n^{3/2})}^{n\rho}
	\left(\frac1r+\frac{\rho}{r^2}\right)\,dr\\
	&=
	\omega_{n-1}
	\left[
	\log(2n^{5/2})
	+2n^{3/2}-\frac1n
	\right]\\
	&\lesssim
	\omega_{n-1}n^{3/2}.
\end{align*}

For \(I_2(y)\), a direct computation yields
\begin{align*}
	I_2(y)
	&\le
	\omega_{n-1}
	\int_{\rho/(2n^{3/2})}^{n\rho}
	\frac{r^{n-1}}{r^n}\,dr\\
	&=
	\omega_{n-1}
	\int_{\rho/(2n^{3/2})}^{n\rho}
	\frac{dr}{r}\\
	&=
	\omega_{n-1}
	\log\bigl(2n^{5/2}\bigr).\\
\end{align*}

The preceding estimates gives
\begin{align*}
I_1(y)+I_2(y)
\lesssim
\omega_{n-1}n^{3/2}.
\end{align*}

Consequently
\begin{align*}
	&\int_F
	\left(
	\int_0^\infty
	\sum_{i=1}^n |B_{i,2}(x,t)|^2\,t\,dt
	\right)^{1/2}dx\\
	&\qquad\le c_{n,\alpha}(n+\alpha)
	\left(\frac12B(\alpha+1,n+1)\right)^{1/2}
	\int_{Q_1}|b_1(y)|
	\,\omega_{n-1}n^{3/2}\,dy\\
	&\qquad\lesssim C_\alpha n^2\|b_1\|_{L^1(\mathbb{R}^n)}.
\end{align*}

Thus, we obtain
\[
\begin{aligned}
	&\int_{F}
	\left(
	\int_0^\infty\sum_{i=1}^n|B_{i,1}(x,t)|^2t\,dt
	\right)^{1/2}dx +
	\int_{F}
	\left(
	\int_0^\infty\sum_{i=1}^n|B_{i,2}(x,t)|^2t\,dt
	\right)^{1/2}dx\\
&	\lesssim C_\alpha n^2\|b_1\|_{L^1(\mathbb{R}^n)}.
\end{aligned}
\]
\subsection{Estimate of $D$}

We split $D$ as
\[
D(x,t) = D_1(x,t) + D_2(x,t),
\]
where $D_1$ and $D_2$ correspond to the regions $|x-y| > n|y-y_0|$ and $|x-y| \leq n|y-y_0|$, respectively.

\medskip

\noindent \textbf{Estimate of $D_1$.}
For $x\in (Q_1^*)^c$, $y\in Q_1$, and $|x-y| > n|y-y_0|$, we have
\[
|x-y| > \frac{n}{n+1}|x-y_0|, \qquad |x-y_0| \geq \frac{a}{2}, \qquad |y-y_0| \leq \sqrt{n} a.
\]

By the mean value theorem
\[
\begin{aligned}
	|D_1(x,t)|
	&\le Cc_{n,\alpha}\alpha(n+\alpha)a\sqrt n\,
	\|b_1\|_{L^1(\mathbb{R}^n)} \\
	&\quad\times
	\frac{|x-y_0|t^{\alpha-1}}
	{\left(t^2+\left(\frac{n}{n+1}|x-y_0|\right)^2\right)^{(n+\alpha+2)/2}}.
\end{aligned}
\]

Furthermore, by (\ref{eq:3.16}), we derive
\[
\begin{aligned}
	&\left(
	\int_0^\infty
	\frac{t^{2\alpha-1}}
	{\left(t^2+\left(\frac{n}{n+1}|x-y_0|\right)^2\right)^{n+\alpha+2}}\,dt
	\right)^{1/2}\\
	&=\left(\frac12B(\alpha,n+2)\right)^{1/2}
	\left(\frac{n+1}{n}\right)^{n+2}|x-y_0|^{-n-2}\\
	&\le C\left(\frac12B(\alpha,n+2)\right)^{1/2}|x-y_0|^{-n-2}.
\end{aligned}
\]

Thus, we obtain
\[
\begin{aligned}
	&\int_F
	\left(\int_0^\infty|D_1(x,t)|^2t\,dt\right)^{1/2}dx\\
	&\le Cc_{n,\alpha}\omega_{n-1}\alpha(n+\alpha)\sqrt n
	\left(\frac12B(\alpha,n+2)\right)^{1/2}
	\|b_1\|_{L^1(\mathbb{R}^n)} \\
	&\lesssim C_\alpha n^{3/2}\|b_1\|_{L^1(\mathbb{R}^n)} .
\end{aligned}
\]

\medskip

\noindent \textbf{Estimate of $D_2$.}
In the region $|x-y| \leq n|y-y_0|$, we have
\[
\frac{|y-y_0|}{2n^{3/2}} \leq |x-y| \leq n|y-y_0|,
\qquad
\frac{|y-y_0|}{2n^{3/2}} \leq |x-y_0| \leq (n+1)|y-y_0|.
\]

Minkowski's inequality, Fubini's theorem and (\ref{eq:3.17}) gives
\[
\begin{aligned}
	&\left(\int_0^\infty|D_2(x,t)|^2t\,dt\right)^{1/2}\\
	&\le c_{n,\alpha}\alpha
	\left(\frac12B(\alpha,n)\right)^{1/2}
	\int_{\substack{y\in Q_1\\|x-y|\le n|y-y_0|}}
	|b_1(y)|
	\left[|x-y|^{-n}+|x-y_0|^{-n}\right]dy.
\end{aligned}
\]

The two spatial integrals satisfy
\[
\int_{\frac{|y-y_0|}{2n^{3/2}}\le|x-y|\le n|y-y_0|}
|x-y|^{-n}\,dx
=\omega_{n-1}\log(2n^{5/2})
\]
and
\[
\int_{\frac{|y-y_0|}{2n^{3/2}}\le|x-y_0|\le(n+1)|y-y_0|}
|x-y_0|^{-n}\,dx
=\omega_{n-1}\log\bigl(2(n+1)n^{3/2}\bigr).
\]

Therefore
\[
\begin{aligned}
	&\int_{(Q_1^*)^c}
	\left(\int_0^\infty|D_2(x,t)|^2t\,dt\right)^{1/2}dx\\
	&\le Cc_{n,\alpha}\omega_{n-1}\alpha
	\left(\frac12B(\alpha,n)\right)^{1/2}
	\log(en)\|b_1\|_{L^1(\mathbb{R}^n)} \\
	&\lesssim C_\alpha\log(en)\|b_1\|_{L^1(\mathbb{R}^n)} .
\end{aligned}
\]

Thus, we obtain
\[
	\begin{aligned}
		&\int_F
		\left(\int_0^\infty|D_1(x,t)|^2t\,dt\right)^{1/2}dx +
		\int_F
		\left(\int_0^\infty|D_2(x,t)|^2t\,dt\right)^{1/2}dx\\
		&\lesssim C_\alpha n^{3/2}\|b_1\|_{L^1(\mathbb{R}^n)} .
\end{aligned}
\]

\subsection{Estimate of $E$}

Finally, we split $E$ as
\[
E(x,t) = E_1(x,t) + E_2(x,t),
\]
with $E_1$ corresponding to $|x-y| > n|y-y_0|$ and $E_2$ to $|x-y| \leq n|y-y_0|$.

\medskip

\noindent \textbf{Estimate of $E_1$.}
In the region $|x-y| > n|y-y_0|$, we have
\[
|x-y| > \frac{n}{n+1}|x-y_0|, \qquad |x-y_0| \geq \frac{a}{2}, \qquad |y-y_0| \leq \sqrt{n} a.
\]

By the mean value theorem
\[
\begin{aligned}
	|E_1(x,t)|
	&\le Cc_{n,\alpha}(n+\alpha)(n+\alpha+2)a\sqrt n\,
	\|b_1\|_{L^1(\mathbb{R}^n)} \\
	&\quad\times
	\frac{|x-y_0|t^{\alpha+1}}
	{\left(t^2+\left(\frac{n}{n+1}|x-y_0|\right)^2\right)^{(n+\alpha+4)/2}}.
\end{aligned}
\]

Moreover, by (\ref{eq:3.18}), we derive
\[
\begin{aligned}
	&\left(
	\int_0^\infty
	\frac{t^{2\alpha+3}}
	{\left(t^2+\left(\frac{n}{n+1}|x-y_0|\right)^2\right)^{n+\alpha+4}}\,dt
	\right)^{1/2}\\
	&=\left(\frac12B(\alpha+2,n+2)\right)^{1/2}
	\left(\frac{n+1}{n}\right)^{n+2}|x-y_0|^{-n-2}\\
	&\le C\left(\frac12B(\alpha+2,n+2)\right)^{1/2}|x-y_0|^{-n-2}.
\end{aligned}
\]

Consequently
\[
\begin{aligned}
	&\int_F
	\left(\int_0^\infty|E_1(x,t)|^2t\,dt\right)^{1/2}dx\\
	&\le Cc_{n,\alpha}\omega_{n-1}(n+\alpha)(n+\alpha+2)\sqrt n
	\left(\frac12B(\alpha+2,n+2)\right)^{1/2}
	\|b_1\|_{L^1(\mathbb{R}^n)} \\
	&\lesssim C_\alpha n^{3/2}\|b_1\|_{L^1(\mathbb{R}^n)} .
\end{aligned}
\]

\medskip

\noindent \textbf{Estimate of $E_2$.}
Using the same argument as for $B_{i,2}$, $D_2$ and (\ref{eq:3.19}), we have
\[
\begin{aligned}
	&\left(\int_0^\infty|E_2(x,t)|^2t\,dt\right)^{1/2}\\
	&\le c_{n,\alpha}(n+\alpha)
	\left(\frac12B(\alpha+2,n)\right)^{1/2}
	\int_{\substack{y\in Q_1\\|x-y|\le n|y-y_0|}}
	|b_1(y)|
	\left[|x-y|^{-n}+|x-y_0|^{-n}\right]dy.
\end{aligned}
\]

Using the same two logarithmic spatial integrals as above, we find
\[
\begin{aligned}
	&\int_F
	\left(\int_0^\infty|E_2(x,t)|^2t\,dt\right)^{1/2}dx\\
	&\le Cc_{n,\alpha}\omega_{n-1}(n+\alpha)
	\left(\frac12B(\alpha+2,n)\right)^{1/2}
	\log(en)\|b_1\|_{L^1(\mathbb{R}^n)} \\
	&\lesssim C_\alpha\log(en)\|b_1\|_{L^1(\mathbb{R}^n)} .
\end{aligned}
\]

Thus, we obtain
\[
	\begin{aligned}
		&\int_F
		\left(\int_0^\infty|E_1(x,t)|^2t\,dt\right)^{1/2}dx	+
		\int_F
		\left(\int_0^\infty|E_2(x,t)|^2t\,dt\right)^{1/2}dx\\
		&\lesssim C_\alpha n^{3/2}\|b_1\|_{L^1(\mathbb{R}^n)} .
\end{aligned}
\]

\medskip

Combining the estimates for $A_i$, $B_i$, $D$, and $E$ obtained above, we conclude that
\begin{equation}\label{eq:final}
	\int_F g_\alpha(b_1)(x)\,dx
	\lesssim
	C_\alpha n^{2} \|b_1\|_{L^1(\mathbb{R}^n)} .
\end{equation}

Summing \eqref{eq:final} over all $k$ and applying \eqref{eq:badsum}, we obtain
\[
\left|\left\{x\in F : g_\alpha(b)(x) > \lambda/2 \right\}\right|
\lesssim \frac{C_\alpha n^{2}}{\lambda} \|f\|_{L^1(\mathbb{R}^n)}.
\]
This establishes \eqref{eq:main}. Combining this with \eqref{eq:good},  \eqref{eq:splitH}, and \eqref{eq:H} completes the proof of Theorem \ref{thm:weak11}.
\end{proof}

\section{Proof of Theorem \ref{thm:limit}}

In this section, we establish the limiting weak-type behavior of $g_\alpha$ as the level parameter $\lambda$ tends to zero. The proof proceeds via a scaling argument combined with a decomposition of $f$ into its near and far parts, and sharp point-wise asymptotics of the gradient kernel.

\begin{proof}
Let $\lambda>0$ and set $s := \lambda^{1/n}$. For $f\in L^1(\mathbb{R}^n)$, we define the scaled function
\[
f_s(y) := \frac{1}{s^n} f\left(\frac{y}{s}\right).
\]

A direct scaling argument yields
\[
g_\alpha(f_s)(x) = \frac{1}{s^n} g_\alpha(f)\left(\frac{x}{s}\right).
\]

Indeed, this follows from the fact that the Poisson kernel $P_t^\alpha$ and the measure $t\,dt$ scale homogeneously. Consequently
\[
\lambda \left|\left\{x\in \mathbb{R}^n : g_\alpha(f)(x) > \lambda \right\}\right|
= \left|\left\{x\in \mathbb{R}^n : g_\alpha(f_s)(x) > 1 \right\}\right|.
\]

Denote
\[
D^s := \left\{x\in \mathbb{R}^n : g_\alpha(f_s)(x) > 1 \right\}.
\]

Since $\lambda \to 0^+$ if and only if $s \to 0^+$, it suffices to prove the following limit
\begin{equation}\label{eq:limitgoal}
	\lim_{s\to 0^+} |D^s|
	= \frac{c_{n,\alpha} \omega_{n-1}}{n}
	\sqrt{\frac{n!}{2\prod_{k=1}^{n-1}(k+\alpha)}}
	\left|\int_{\mathbb{R}^n} f(x)\,dx\right|.
\end{equation}

Without loss of generality, we assume that $\|f\|_{L^1(\mathbb{R}^n)} = 1$. For $0<\delta<1/5$, the absolute continuity of the Lebesgue integral implies that for every $0<\epsilon<\delta/2$, there exists $a_\epsilon>0$ such that
\[
\int_{B(0,a_\epsilon)} |f(x)|\,dx = 1 - \epsilon.
\]

Set
\[
f^0 := f \chi_{B(0,a_\epsilon)}, \qquad f^\infty := f - f^0.
\]

Let $f_s^0$ and $f_s^\infty$ denote the scaled versions of $f^0$ and $f^\infty$, respectively, i.e.,
\[
f_s^0(y) := \frac{1}{s^n} f^0\left(\frac{y}{s}\right), \qquad
f_s^\infty(y) := \frac{1}{s^n} f^\infty\left(\frac{y}{s}\right).
\]

Then $f_s = f_s^0 + f_s^\infty$, and we have the $L^1$-norm identities
\[
\|f_s^0\|_{L^1(\mathbb{R}^n)} = 1 - \epsilon, \qquad
\|f_s^\infty\|_{L^1(\mathbb{R}^n)} = \epsilon.
\]

For any $\gamma>0$, define the superlevel sets
\[
E_\gamma^s := \left\{x\in \mathbb{R}^n : g_\alpha(f_s^\infty)(x) > \gamma \right\},
\qquad
F_\gamma^s := \left\{x\in \mathbb{R}^n : g_\alpha(f_s^0)(x) > \gamma \right\}.
\]

Since $g_\alpha$ is subadditive, we have the inclusions
\[
F_{1+\delta}^s \subset E_\delta^s \cup D^s,
\qquad
D^s \subset E_\delta^s \cup F_{1-\delta}^s.
\]

Consequently,
\begin{equation}\label{eq:sandwich}
	|F_{1+\delta}^s| - |E_\delta^s|
	\leq |D^s|
	\leq |E_\delta^s| + |F_{1-\delta}^s|.
\end{equation}

By Theorem \ref{thm:weak11}, we have
\[
|E_\delta^s|
= \left|\left\{x\in \mathbb{R}^n : g_\alpha(f_s^\infty)(x) > \delta \right\}\right|
\leq \frac{C_{n,\alpha}}{\delta} \|f_s^\infty\|_{L^1(\mathbb{R}^n)}
= \frac{C_{n,\alpha} \epsilon}{\delta},
\]
where $C_{n,\alpha}$ is a constant depending on $n$ and $\alpha$.

It remains to estimate $|F_{1-\delta}^s|$. To this end, we employ Minkowski's inequality. For $x\in \mathbb{R}^n$, we have
\[
g_\alpha(f_s^0)(x)
\leq H(f_s^0)(x) + J_\alpha(x) \left|\int_{\mathbb{R}^n} f_s^0(y)\,dy\right|,
\]
where
\[
H(f_s^0)(x) := \left\| \int_{\mathbb{R}^n} \nabla \big[ P_t^\alpha(x-y) - P_t^\alpha(x) \big] f_s^0(y)\,dy \right\|_{L^2(tdt)},
\]

and
\[
J_\alpha(x) := \left( \int_0^\infty |\nabla P_t^\alpha(x)|^2 t\,dt \right)^{1/2}.
\]

Similarly, we have the corresponding lower bound
\[
g_\alpha(f_s^0)(x)
\geq J_\alpha(x) \left|\int_{\mathbb{R}^n} f_s^0(y)\,dy\right| - H(f_s^0)(x).
\]

Choose $\eta > 2s a_\epsilon$ and define
\[
G_s := \left\{ x\in B(0,\eta)^c : H(f_s^0)(x) > \delta \right\}.
\]

With this definition, we clearly have
\[
|F_{1-\delta}^s|
\leq |B(0,\eta)| + |G_s| + |F_{1-\delta}^s \cap B(0,\eta)^c \cap G_s^c|.
\]

We first estimate $|G_s|$. By Chebyshev's inequality and Minkowski's inequality, we obtain

	\begin{align*}
|G_s|&
\leq \frac{1}{\delta} \int_{B(0,\eta)^c} H(f_s^0)(x)\,dx\\
&\leq \frac{1}{\delta} \int_{B(0,\eta)^c} \int_{\mathbb{R}^n} |f_s^0(y)|
\left\| \nabla \big[ P_t^\alpha(x-y) - P_t^\alpha(x) \big] \right\|_{L^2(tdt)} \,dy\,dx.
\end{align*}

Since $|x| > \eta > 2s a_\epsilon > 2|y|$ for $y\in \operatorname{supp}(f_s^0) \subset B(0,s a_\epsilon)$, we have
\[
\frac{1}{2} |x| < |x-y| < \frac{3}{2} |x|.
\]

A standard gradient estimate  gives
\[
\left\| \nabla \big[ P_t^\alpha(x-y) - P_t^\alpha(x) \big] \right\|_{L^2(tdt)}
\leq C c_{n,\alpha} 2^n n^{7/2} \frac{|y|}{|x|^{n+1}}.
\]

Hence, in view of the above estimate, we conclude
\begin{equation}\label{eq:Gsbound}
	\begin{aligned}
		|G_s|
		&\leq \frac{C c_{n,\alpha} 2^n n^{7/2}}{\delta}
		\int_{B(0,\eta)^c} \int_{B(0,s a_\epsilon)} |f_s^0(y)| |y|\,dy \frac{dx}{|x|^{n+1}} \\
		&\leq \frac{C_{n,\alpha} s a_\epsilon}{\delta} \int_{B(0,s a_\epsilon)} |f_s^0(y)|\,dy \int_{B(0,\eta)^c} \frac{dx}{|x|^{n+1}} \\
		&\leq C_{n,\alpha}\frac{s a_\epsilon}{\delta \eta}.
	\end{aligned}
\end{equation}

It remains to estimate the set
\[
F_{1-\delta}^s \cap B(0,\eta)^c \cap G_s^c.
\]

For $x\neq 0$, $P_t^\alpha(x)$ is smooth with respect to $t\in (0,\infty)$. Moreover, the gradient $\nabla P_t^\alpha(x)$ is given explicitly by
\[
\nabla P_t^\alpha(x)
=
\left(
\frac{c_{n,\alpha} t^{\alpha-1}(\alpha |x|^2 - n t^2)}{(t^2 + |x|^2)^{\frac{n+\alpha}{2} + 1}},
-
\frac{c_{n,\alpha}(n+\alpha) t^\alpha x_1}{(t^2 + |x|^2)^{\frac{n+\alpha}{2} + 1}},
\dots,
-
\frac{c_{n,\alpha}(n+\alpha) t^\alpha x_n}{(t^2 + |x|^2)^{\frac{n+\alpha}{2} + 1}}
\right).
\]

Using the elementary integral formula
\[
\int_0^\infty \frac{t^k}{(t^2 + r^2)^m}\,dt
= \frac{\Gamma\left(\frac{k+1}{2}\right) \Gamma\left(m - \frac{k+1}{2}\right)}{2 r^{2m-k-1} \Gamma(m)},
\]
valid for $k>-1$ and $m>(k+1)/2$, we compute
\begin{equation}\label{eq:Jexplicit}
	J_\alpha(x)
	= \left\| \nabla P_t^\alpha(x) \right\|_{L^2(tdt)}
	= \frac{c_{n,\alpha}}{|x|^n}
	\sqrt{\frac{n!}{2\prod_{k=1}^{n-1}(k+\alpha)}}.
\end{equation}

For $x\in B(0,\eta)^c \cap G_s^c$, we have $H(f_s^0)(x) \leq \delta$. Combining the upper and lower bounds for $g_\alpha(f_s^0)(x)$, we obtain
\[
J_\alpha(x) \left|\int_{\mathbb{R}^n} f_s^0(y)\,dy\right| - \delta
\leq g_\alpha(f_s^0)(x)
\leq J_\alpha(x) \left|\int_{\mathbb{R}^n} f_s^0(y)\,dy\right| + \delta.
\]

Note that for $x\in F_{1-\delta}^s$, it follows that $g_\alpha(f_s^0)(x) > 1-\delta$. Therefore, combining with the right-hand side of the above inequality, we have
\[
F_{1-\delta}^s \cap B(0,\eta)^c \cap G_s^c
\subseteq
\left\{ x\in B(0,\eta)^c \cap G_s^c :
J_\alpha(x) \left|\int_{\mathbb{R}^n} f_s^0(y)\,dy\right| > 1 - 2\delta \right\}.
\]

Recall the elementary measure formula
\[
\left|\left\{ x\in \mathbb{R}^n : \frac{1}{|x|^n} > \mu \right\}\right|
= \frac{\omega_{n-1}}{n\mu}, \qquad \mu>0.
\]

Then by the above inclusion, this measure formula, and \eqref{eq:Jexplicit}, we obtain
\begin{equation}\label{eq:Fupper}
	\begin{aligned}
		&|F_{1-\delta}^s \cap B(0,\eta)^c \cap G_s^c| \\
		&\qquad \leq
		\left|\left\{ x\in \mathbb{R}^n :
		J_\alpha(x) \left|\int_{\mathbb{R}^n} f_s^0(y)\,dy\right| > 1 - 2\delta \right\}\right| \\
		&\qquad =
		\left|\left\{ x\in \mathbb{R}^n :
		\frac{c_{n,\alpha}}{|x|^n}
		\sqrt{\frac{n!}{2\prod_{k=1}^{n-1}(k+\alpha)}}
		\left|\int_{\mathbb{R}^n} f_s^0(y)\,dy\right| > 1 - 2\delta \right\}\right| \\
		&\qquad \leq
		\frac{c_{n,\alpha} \omega_{n-1}}{n(1-2\delta)}
		\sqrt{\frac{n!}{2\prod_{k=1}^{n-1}(k+\alpha)}}
		\left( \left|\int_{\mathbb{R}^n} f(y)\,dy\right| + \epsilon \right),
	\end{aligned}
\end{equation}
where in the last step we have used the fact that
\[
\left|\int_{\mathbb{R}^n} f(y)\,dy\right| - \epsilon
\leq
\left|\int_{\mathbb{R}^n} f_s^0(y)\,dy\right|
\leq
\left|\int_{\mathbb{R}^n} f(y)\,dy\right| + \epsilon.
\]

Next, we consider the left-hand side of \eqref{eq:sandwich}. The set $F_{1+\delta}^s$ can be treated in the same manner. Clearly
\[
|F_{1+\delta}^s| \geq |F_{1+\delta}^s \cap B(0,\eta)^c \cap G_s^c|.
\]

Observe that for $x\in B(0,\eta)^c \cap G_s^c$ satisfying
\[
J_\alpha(x) \left|\int_{\mathbb{R}^n} f_s^0(y)\,dy\right| > 1 + 2\delta,
\]
the lower bound for $g_\alpha(f_s^0)(x)$ implies that $g_\alpha(f_s^0)(x) > 1+\delta$. Therefore
\[
F_{1+\delta}^s \cap B(0,\eta)^c \cap G_s^c
\supseteq
\left\{ x\in B(0,\eta)^c \cap G_s^c :
J_\alpha(x) \left|\int_{\mathbb{R}^n} f_s^0(y)\,dy\right| > 1 + 2\delta \right\}.
\]

Moreover, using \eqref{eq:Jexplicit} and \eqref{eq:Gsbound}, we obtain
\begin{equation}\label{eq:Flower}
	\begin{aligned}
		&|F_{1+\delta}^s \cap B(0,\eta)^c \cap G_s^c| \\
		&\qquad \geq
		\left|\left\{ x\in \mathbb{R}^n :
		J_\alpha(x) \left|\int_{\mathbb{R}^n} f_s^0(y)\,dy\right| > 1 + 2\delta \right\}\right|
		- |B(0,\eta)| - |G_s| \\
		&\qquad \geq
		\frac{c_{n,\alpha} \omega_{n-1}}{n(1+2\delta)}
		\sqrt{\frac{n!}{2\prod_{k=1}^{n-1}(k+\alpha)}}
		\left( \left|\int_{\mathbb{R}^n} f(y)\,dy\right| - \epsilon \right)
		- \frac{\omega_{n-1}\eta^n}{n}
		- C_{n,\alpha} \frac{s a_\epsilon}{\delta \eta}.
	\end{aligned}
\end{equation}

To this end, in view of estimates \eqref{eq:sandwich}, \eqref{eq:Gsbound}, and \eqref{eq:Fupper}, we conclude
\begin{equation}\label{eq:Dsupper}
	\begin{aligned}
		|D^s|
		\leq& |E_\delta^s| + |B(0,\eta)| + |G_s| + |F_{1-\delta}^s \cap B(0,\eta)^c \cap G_s^c| \\
		\leq &\frac{C_{n,\alpha} \epsilon}{\delta}
		+ \frac{\omega_{n-1}\eta^n}{n}
		+ C_{n,\alpha} \frac{s a_\epsilon}{\delta \eta}\\
		&+ \frac{c_{n,\alpha} \omega_{n-1}}{n(1-2\delta)}
		\sqrt{\frac{n!}{2\prod_{k=1}^{n-1}(k+\alpha)}}
		\left( \left|\int_{\mathbb{R}^n} f(y)\,dy\right| + \epsilon \right),
	\end{aligned}
\end{equation}
which implies
\[
\begin{aligned}
	\limsup_{s\to 0^+} |D^s|
	&\leq \frac{C_{n,\alpha} \epsilon}{\delta}
	+ \frac{\omega_{n-1}\eta^n}{n}
	+ \frac{c_{n,\alpha} \omega_{n-1}}{n(1-2\delta)}
	\sqrt{\frac{n!}{2\prod_{k=1}^{n-1}(k+\alpha)}}
	\left( \left|\int_{\mathbb{R}^n} f(y)\,dy\right| + \epsilon \right).
\end{aligned}
\]

Note that $\epsilon$ is arbitrarily small relative to the other constants. Letting $\epsilon \to 0^+$, then $\delta \to 0^+$, and $\eta \to 0^+$, we obtain
\begin{equation}\label{eq:limsupfinal}
	\limsup_{s\to 0^+} |D^s|
	\leq
	\frac{c_{n,\alpha} \omega_{n-1}}{n}
	\sqrt{\frac{n!}{2\prod_{k=1}^{n-1}(k+\alpha)}}
	\left|\int_{\mathbb{R}^n} f(y)\,dy\right|.
\end{equation}

Similarly, in view of estimates \eqref{eq:sandwich}, \eqref{eq:Gsbound}, and \eqref{eq:Flower}, we conclude
\begin{equation}\label{eq:Dslower}
	\begin{aligned}
		|D^s|
		\geq &|F_{1+\delta}^s \cap B(0,\eta)^c \cap G_s^c| - |E_\delta^s| - |B(0,\eta)| - |G_s| \\
		\geq&
		\frac{c_{n,\alpha} \omega_{n-1}}{n(1+2\delta)}
		\sqrt{\frac{n!}{2\prod_{k=1}^{n-1}(k+\alpha)}}
		\left( \left|\int_{\mathbb{R}^n} f(y)\,dy\right| - \epsilon \right)\\
		&- \frac{C_{n,\alpha} \epsilon}{\delta}
		- \frac{2\omega_{n-1}\eta^n}{n}
		- 2C_{n,\alpha} \frac{s a_\epsilon}{\delta \eta}.
	\end{aligned}
\end{equation}

Letting $s\to 0^+$, then $\epsilon\to 0^+$, $\delta\to 0^+$, and $\eta\to 0^+$, we get
\begin{equation}\label{eq:liminffinal}
	\liminf_{s\to 0^+} |D^s|
	\geq
	\frac{c_{n,\alpha} \omega_{n-1}}{n}
	\sqrt{\frac{n!}{2\prod_{k=1}^{n-1}(k+\alpha)}}
	\left|\int_{\mathbb{R}^n} f(y)\,dy\right|.
\end{equation}

Combining \eqref{eq:limsupfinal} and \eqref{eq:liminffinal}, we obtain
\[
\lim_{s\to 0^+} |D^s|
=
\frac{c_{n,\alpha} \omega_{n-1}}{n}
\sqrt{\frac{n!}{2\prod_{k=1}^{n-1}(k+\alpha)}}
\left|\int_{\mathbb{R}^n} f(x)\,dx\right|.
\]

This proves \eqref{eq:limitgoal} and hence completes the proof of Theorem \ref{thm:limit}.

\end{proof}

\appendix
\section{The derivative of the modified Bessel function}
To the best of our knowledge, although the derivative formula for $K_\nu$  is known, a comprehensive proof has not appeared in the literature. We therefore supply a complete proof for the convenience of the reader.

\begin{lemma}\label{derivative}
For $K_\nu$ defined by $$K_\nu(z)= \frac{\pi i}{2}\,e^{\frac{\pi}{2}\nu i}\; H_\nu^{(1)}\!\bigl(z e^{\frac{\pi}{2}i}\bigr),H_\nu^{(1)}(z)= -\frac{i}{\pi}\,e^{-\frac{1}{2}i\nu\pi}
	\int_0^\infty \exp\!\Bigl(\frac{1}{2}iz\,(t+t^{-1})\Bigr)\,t^{-\nu-1}\,dt,$$
the following differentiation formulas holds:
\begin{align*}
	K_\nu'(z)=-\frac{\nu}{z}K_\nu(z)-K_{\nu-1}(z)\;, \;(s^\nu K_\nu(s))' = -s^\nu K_{\nu-1}(s).
\end{align*}
\end{lemma}	
\begin{proof}
	Since the second formula follows readily from the first, we only prove the former. Replace $z$ by $iz$,
	\begin{align*}
		H_\nu^{(1)}(i z)=&-\frac{i}{\pi}e^{-\frac{1}{2}i\nu\pi}
		\int_0^\infty \exp\!\Bigl(\frac{1}{2}i(i z)(t+t^{-1})\Bigr)\,t^{-\nu-1}dt\\
		=&-\frac{i}{\pi}e^{-\frac{1}{2}i\nu\pi}\int_0^\infty e^{-\frac{z}{2}(t+t^{-1})}\,t^{-\nu-1}dt.
	\end{align*}	

	Then
	\begin{align*}
		K_\nu(z)&=\frac{\pi i}{2}e^{\frac{\pi}{2}\nu i}\; H_\nu^{(1)}(i z)\\
		&=\frac{\pi i}{2}e^{\frac{\pi}{2}\nu i}\Bigl(-\frac{i}{\pi}e^{-\frac{1}{2}i\nu\pi}\Bigr)
		\int_0^\infty e^{-\frac{z}{2}(t+t^{-1})}t^{-\nu-1}dt.
	\end{align*}

%	The constant factor simplifies: $\frac{\pi i}{2}\cdot(-\frac{i}{\pi})=\frac12$, and the exponential factors combine to $e^{\frac{\pi}{2}\nu i}e^{-\frac{1}{2}i\nu\pi}=1$.  Hence
By a straightforward computation, we get
	\begin{equation}\label{KVZ}
			\begin{aligned}
			K_\nu(z)=\frac12\int_0^\infty e^{-\frac{z}{2}(t+t^{-1})}\,t^{-\nu-1}dt.
		\end{aligned}
	\end{equation}

Hence
	\begin{align*}
		K_\nu'(z)&=\frac12\int_0^\infty\frac{\partial}{\partial z}\Bigl(e^{-\frac{z}{2}(t+t^{-1})}\Bigr)t^{-\nu-1}dt\\
		&=\frac12\int_0^\infty\Bigl(-\frac12(t+t^{-1})\Bigr)e^{-\frac{z}{2}(t+t^{-1})}t^{-\nu-1}dt\\
		&=-\frac14\int_0^\infty(t+t^{-1})e^{-\frac{z}{2}(t+t^{-1})}t^{-\nu-1}dt\\
		&=-\frac14\Bigl(\int_0^\infty t\,e^{-\frac{z}{2}(t+t^{-1})}t^{-\nu-1}dt+\int_0^\infty t^{-1}e^{-\frac{z}{2}(t+t^{-1})}t^{-\nu-1}dt\Bigr)\\
		&=-\frac14\Bigl(\int_0^\infty t^{-\nu}e^{-\frac{z}{2}(t+t^{-1})}dt+\int_0^\infty t^{-\nu-2}e^{-\frac{z}{2}(t+t^{-1})}dt\Bigr).
	\end{align*}

	From (\ref{KVZ}), we have
	\begin{align*}
		K_{\nu-1}(z)&=\frac12\int_0^\infty e^{-\frac{z}{2}(t+t^{-1})}t^{-(\nu-1)-1}dt=\frac12\int_0^\infty t^{-\nu}e^{-\frac{z}{2}(t+t^{-1})}dt,\\
		K_{\nu+1}(z)&=\frac12\int_0^\infty e^{-\frac{z}{2}(t+t^{-1})}t^{-(\nu+1)-1}dt=\frac12\int_0^\infty t^{-\nu-2}e^{-\frac{z}{2}(t+t^{-1})}dt.
	\end{align*}	

Consequently
	\begin{align*}
		\int_0^\infty t^{-\nu}e^{-\frac{z}{2}(t+t^{-1})}dt&=2K_{\nu-1}(z),\\
		\int_0^\infty t^{-\nu-2}e^{-\frac{z}{2}(t+t^{-1})}dt&=2K_{\nu+1}(z).
	\end{align*}

	Therefore
	\begin{equation}\label{ire}
		\begin{aligned}
			K_\nu'(z)=-\frac14\bigl(2K_{\nu-1}(z)+2K_{\nu+1}(z)\bigr)=-\frac12\bigl(K_{\nu-1}(z)+K_{\nu+1}(z)\bigr).
		\end{aligned}
	\end{equation}
	
	We next consider $K_{\nu-1}(z)-K_{\nu+1}(z)$.
	\begin{align*}
		K_{\nu-1}-K_{\nu+1}=\frac12\int_0^\infty e^{-\frac{z}{2}(t+t^{-1})}\bigl(t^{-\nu}-t^{-\nu-2}\bigr)dt
		=\frac12\int_0^\infty e^{-\frac{z}{2}(t+t^{-1})}t^{-\nu-1}(t-t^{-1})dt.
	\end{align*}

Note that
	\begin{align*}
		(t-t^{-1})e^{-\frac{z}{2}(t+t^{-1})}=-\frac{2}{z}\,t\,\frac{d}{dt}e^{-\frac{z}{2}(t+t^{-1})}.
	\end{align*}

	Inserting this into $K_{\nu-1}(z)-K_{\nu+1}(z)$ gives
	\begin{align*}
		K_{\nu-1}-K_{\nu+1}=\frac12\int_0^\infty t^{-\nu-1}\Bigl(-\frac{2}{z}t\frac{d}{dt}e^{-\frac{z}{2}(t+t^{-1})}\Bigr)dt
		=-\frac{1}{z}\int_0^\infty t^{-\nu}\frac{d}{dt}e^{-\frac{z}{2}(t+t^{-1})}dt.
	\end{align*}

	Integrating by parts, we have
	\begin{align*}
		\int_0^\infty t^{-\nu}\frac{d}{dt}e^{-\frac{z}{2}(t+t^{-1})}dt
		&=\Bigl[t^{-\nu}e^{-\frac{z}{2}(t+t^{-1})}\Bigr]_0^\infty-\int_0^\infty e^{-\frac{z}{2}(t+t^{-1})}\frac{d}{dt}(t^{-\nu})dt\\
		&=0-(-\nu\int_0^\infty e^{-\frac{z}{2}(t+t^{-1})}t^{-\nu-1}dt)\\
		&=\nu\int_0^\infty e^{-\frac{z}{2}(t+t^{-1})}t^{-\nu-1}dt.
	\end{align*}

	Hence
	\begin{equation}\label{rel}
		\begin{aligned}
			K_{\nu-1}(z)-K_{\nu+1}(z)&=-\frac{1}{z}\nu\int_0^\infty e^{-\frac{z}{2}(t+t^{-1})}t^{-\nu-1}dt&=-\frac{2\nu}{z}K_\nu(z).
		\end{aligned}	
	\end{equation}
	
	From (\ref{ire}) and (\ref{rel}) we have
	\begin{align*}
		K_{\nu-1}(z)+K_{\nu+1}(z)=-2K_\nu'(z),\qquad
		K_{\nu-1}(z)-K_{\nu+1}(z)=-\frac{2\nu}{z}K_\nu(z).
	\end{align*}

	Adding the two equations and dividing by $2$ gives
	\begin{align*}
		K_{\nu-1}(z) = -K_\nu'(z)-\frac{\nu}{z}K_\nu(z),
	\end{align*}
	which is precisely the desired identity
	\begin{align*}
		K_\nu'(z)=-\frac{\nu}{z}K_\nu(z)-K_{\nu-1}(z).
	\end{align*}
\end{proof}

\end{document}